\documentclass[12pt,leqno]{amsart}

\usepackage{latexsym,amsthm,amsmath,amssymb,mathrsfs}
\usepackage{graphicx}
\usepackage{color}

\usepackage{setspace}

\title[Homeomorphism with negative Jacobian]{A measure and orientation preserving homeomorphism with approximate Jacobian equal $-1$ almost everywhere}
\author{Pawe\l{} Goldstein}
\address{Pawe\l{} Goldstein\newline \indent Institute of Mathematics \newline \indent Faculty of Mathematics, Informatics and Mechanics\newline \indent University of Warsaw
\newline \indent Banacha 2\newline \indent 02-097 Warsaw, Poland\newline \indent {\tt goldie@mimuw.edu.pl}}
\author{Piotr Haj\l{}asz}
\address{Piotr Haj{\l}asz\newline \indent Department of Mathematics\newline \indent University of Pittsburgh\newline \indent 301
  Thackeray Hall\newline \indent Pittsburgh, PA 15260, USA\newline \indent {\tt hajlasz@pitt.edu}}
\thanks{P.G. was supported by FNP grant POMOST BIS/2012-6/3 \newline P.H.\ was supported by NSF grant DMS-1500647.}


plus 6pt minus 12pt
\def\eps{\varepsilon}

\def\id{{\rm id\, }}

\def\C{{\mathcal C}}

\newtheorem{theorem}{Theorem}
\newtheorem{lemma}[theorem]{Lemma}
\newtheorem{corollary}[theorem]{Corollary}
\newtheorem{proposition}[theorem]{Proposition}

\def\diam{{\rm diam\,}}
\def\dist{{\rm dist\,}}

\theoremstyle{definition}
\newtheorem{remark}[theorem]{Remark}
\newtheorem{definition}[theorem]{Definition}

\newcommand{\barint}{
\rule[.036in]{.12in}{.009in}\kern-.16in \displaystyle\int }

\newcommand{\barcal}{\mbox{$ \rule[.036in]{.11in}{.007in}\kern-.128in\int $}}

\newcommand{\bbbn}{\mathbb N}
\newcommand{\bbbr}{\mathbb R}

\def\ap{\operatorname{ap}}

\def\diam{\operatorname{diam}}

\def\dist{\operatorname{dist}}

\def\mvint_#1{\mathchoice
          {\mathop{\vrule width 6pt height 3 pt depth -2.5pt
                  \kern -8pt \intop}\nolimits_{\kern -3pt #1}}%
          {\mathop{\vrule width 5pt height 3 pt depth -2.6pt
                  \kern -6pt \intop}\nolimits_{#1}}%
          {\mathop{\vrule width 5pt height 3 pt depth -2.6pt
                  \kern -6pt \intop}\nolimits_{#1}}%
          {\mathop{\vrule width 5pt height 3 pt depth -2.6pt
                  \kern -6pt \intop}\nolimits_{#1}}}

\numberwithin{theorem}{section} \numberwithin{equation}{section}

\begin{document}

\begin{abstract}
We construct an almost everywhere approximately
differentiable, orientation and measure preserving homeomorphism of a unit
$n$-dimensional cube onto itself, whose Jacobian is equal
to $-1$ a.e. Moreover we prove that our homeomorphism can be uniformly approximated by
orientation and measure preserving diffeomorphisms.
\end{abstract}

\maketitle

\section{Introduction}

The classical change of variables formula states that if $\Phi:\bbbr^n\supset\Omega\to\bbbr^n$ is a diffeomorphism, then
\begin{equation}
\label{E1}
\int_\Omega g(x)|J_\Phi(x)|\, dx = \int_{\Phi(\Omega)} g(\Phi^{-1}(y))\, dy,
\end{equation}
where $J_\Phi(x)=\det D\Phi(x)$. It is natural to ask how far we can relax the regularity assumptions for $\Phi$ so that the change of variables formula \eqref{E1} remains
valid. A complete answer to this question was provided by Federer \cite{Federer44}, in 1944.

We say that a mapping $\Phi:\Omega\to\bbbr^n$ defined on an open set $\Omega\subset\bbbr^n$ has the {\em Lusin property (N)} if it maps sets of Lebesgue measure zero to sets of Lebesgue
measure zero. Clearly, a mapping $\Phi$ for which \eqref{E1} is true must satisfy the condition (N). Indeed, if $|E|=0$ but $|\Phi(E)|>0$, then taking $g$
to be the characteristic function of $E$, $g=\chi_E$, we obtain zero on the left hand side of \eqref{E1}, but a positive value on the right hand side, which is a contradiction.
Also, in order to define the Jacobian, $\Phi$ must be differentiable, at least in some weak sense.

Recall that for a measurable set $E\subset\bbbr^n$, $x\in\bbbr^n$ is called a \emph{density point} of $E$, if $\lim_{\rho\to 0} |E\cap B(x,\rho)|/|B(x,\rho)|=1$. By the Lebesgue differentiation theorem, a.e. point of $E$ is its density point.
A measurable function $f:E\to\bbbr$ defined on a measurable set $E\subset\bbbr^n$ is said to be \emph{approximately differentiable} at $x\in E$ if there is a measurable set $E_x\subset E$
and a linear function $L:\bbbr^n\to\bbbr$ such that $x$ is a density point of $E_x$ and
$$
\lim_{E_x\ni y\to x}
\frac{|f(y)-f(x)-L(y-x)|}{|y-x|} = 0.
$$
This definition is equivalent to the classical one, see Appendix.
The approximate derivative $L$ (if it exists) is unique and is denoted by ${\ap D}f(x)$. In the case of mappings into $\bbbr^n$,
approximate differentiability means approximate differentiability of each component. The result of Federer \cite{Federer44},
mentioned above, can be stated as follows (see also \cite{FedererBook,HajlaszChange}). By a homeomorphism $\Phi:\Omega\to\bbbr^n$
we mean a homeomorphism onto the image $\Phi(\Omega)$.

\begin{theorem}[Federer]
\label{FedererThm}
Suppose that $\Phi:\Omega\to\bbbr^n$ is a homeomorphism defined on an open set
$\Omega\subset\bbbr^n$ that
has the Lusin property (N) and is approximately differentiable a.e. Then the change of variables formula \eqref{E1}
holds true, where $J_\Phi (x) = \det \ap D\Phi(x)$.
\end{theorem}

This result has further generalizations, \cite{FedererBook}, known as the area and the co-area formulae,
where it is not required that $\Phi$ is continuous or one-to-one. One can even allow $\Phi$
to be a mapping between spaces of different dimensions. However, for the purpose of this paper we will focus on the
case when $\Phi$ is a homeomorphism, as stated in Theorem~\ref{FedererThm}.

The class of mappings that are approximately differentiable a.e. has been studied in \cite{Whitney}, where many
equivalent characterizations were provided. For example, all mappings in the Sobolev space $W^{1,p}$ are
approximately differentiable a.e. (see \cite{HajlaszChange}). Here $W^{1,p}$ stands for the space of functions (or mappings)
in $L^p$ whose weak derivatives are also in $L^p$.
The class of mappings that are approximately differentiable a.e., along with the change of variables, the area and the co-area formulae,
plays a fundamental role in geometric measure theory and its applications to calculus of variations.
It is of particular importance in the approach
to the nonlinear elasticity introduced by Ball \cite{Ball}, and developed for example in \cite{HajlaszKoskela,MullerQY,MullerS,Sverak}.
A natural and important question is how to relate the topological properties of the mapping to the properties of the Jacobian.
For example, it is easy to prove that if a homeomorphism defined on a domain is differentiable a.e. in the classical sense, then the Jacobian cannot change sign,
see \cite[Theorem~5.22]{HenclK}\footnote{Although in the statement of Theorem~5.22 it is assumed that the homeomorphism belongs to $W^{1,1}_{\rm loc}$, this assumption
is never used in the proof.}. Since homeomorphisms in $W^{1,p}$ are differentiable a.e. when $p>n-1$, \cite[Corollary~2.25]{HenclK}, the Jacobian of such a homeomorphism
cannot change sign. However, without assuming the Lusin condition (N) it may happen that the Jacobian equals zero a.e., when $1\leq p<n$, see
\cite{Hencl} and also \cite{Cerny,DHS,LiuM}.

The following questions were asked by Haj\l{}asz in 2001 (see \cite[Section~5.4]{HenclK} and \cite[p. 234]{HenclM}).

\noindent
{\sc Question~1.}
{\em Is it possible to construct a homeomorphism $\Phi:(0,1)^n\to\bbbr^n$ which is approximately differentiable a.e., has the Lusin property (N) and
at the same time $J_\Phi>0$ on a set of positive measure and $J_\Phi<0$ on a set of positive measure?}

\noindent
{\sc Question~2.}
{\em Is it possible to construct a homeomorphism $\Phi:[0,1]^n\to [0,1]^n$ which is approximately differentiable a.e., has the Lusin property (N),
equals to the identity on the boundary (and hence it is sense preserving in the topological sense\footnote{See remarks at the end of the Introduction}), but $J_\Phi<0$ a.e.?}

\noindent
{\sc Question~3.}
{\em Is it possible to construct a homeomorphism $\Phi:(0,1)^n\to\bbbr^n$ of the Sobolev class $W^{1,p}$, $1\leq p<n-1$, such that at the same time
$J_\Phi>0$ on a set of positive measure and $J_\Phi<0$ on a set of positive measure?}

The answer to Question~1 is in the positive. An example of such a homeomorphism has been known to the authors since 2001,
but it has never been published. It is our Lemma~\ref{base}.

This example has an interesting consequence for the change of variables formula. The homeomorphism $\Phi$ from Lemma~\ref{base} is orientation
preserving and it satisfies the assumptions of Theorem~\ref{FedererThm} so the change of variables formula \eqref{E1} is true. However, $|J_\Phi|$
on the left hand side of \eqref{E1} cannot be replaced by $J_\Phi$ despite the fact that $\Phi$ is orientation preserving.

Using an iterative procedure involving Lemma~\ref{base} we can also answer the Question~2
in the positive, but the construction is much more difficult.
In the main result of the paper, Theorem~\ref{main}, we actually prove much more. We construct such a homeomorphism $\Phi$
with $J_\Phi=-1$ a.e. and we prove that our homeomorphism can be uniformly approximated by a sequence of measure and orientation
preserving diffeomorphisms (i.e., with the Jacobian equal $+1$).
The motivation for this result partially stems from the dynamics of measure preserving homeomorphisms \cite{AlpernP}.

Question~3 has also been answered. As was already pointed out, when $p>n-1$, the Jacobian of a Sobolev $W^{1,p}$ homeomorphism $\Phi$ cannot change sign.
The argument used above was based on a.e. differentiability of $\Phi$.
Another argument can be based on the degree theory \cite{FonsecaG}, and the fact that
the Sobolev embedding theorem allows us to control the topological behavior of $\Phi$ on almost all spheres.
The argument can be extended to the case $p=n-1$ (never published), but
it completely fails when $p<n-1$,
and Haj\l{}asz conjectured back in 2001 that in that case a Sobolev homeomorphism can change the sign of the Jacobian. However, Hencl and Mal\'y \cite{HenclM}
proved that when $n=2,3$, $p\geq 1$ or $n\geq 4$,
$p>[n/2]$ (integer part of $n/2$), a Sobolev homeomorphism cannot change the sign of the Jacobian. This time,
instead of topological degree, Hencl and Mal\'y used the notion of linking number.
The case $n\geq 4$, $1\leq p\leq [n/2]$
was left open and
very recently, after a preliminary version of our paper has already been completed (see reference [18] in \cite{HenclV}),
Hencl and Vejnar \cite{HenclV} answered the Question~3 in the positive when $p=1$ and $n\geq 4$ by constructing a $W^{1,1}$
homeomorphism in $\bbbr^n$, $n\geq 4$, whose Jacobian changes sign.
It easily follows that this homeomorphism cannot be approximated by smooth diffeomorphisms in the Sobolev norm (see \cite[Corollary~1.2]{HenclV}).
This is in contract with the case $n=2$ where every Sobolev homeomorphism can be approximated by smooth diffeomorphisms in the Sobolev norm, see
\cite{HenclP,iwaniecko}. The case $n=3$ remains open.

In this paper we focus on Question~2 without assuming Sobolev regularity of the homeomorphism.
On the other hand, in our main result, Theorem~\ref{main}, we are concerned with the Jacobian equal $-1$ a.e.
and our result is interesting from the perspective of dynamics of measure preserving homeomorphisms.

Let us denote by
$$
d(\Phi,\Psi)=\sup_{x\in Q}|\Phi(x)-\Psi(x)|+\sup_{x\in Q}|\Phi^{-1}(x)-\Psi^{-1}(x)|
$$
the {\em uniform metric} in the space of homeomorphisms of the unit cube $Q=[0,1]^n$ onto itself. It turns out that the subspace of measure preserving
homeomorphisms is a complete metric space with respect to the metric $d$. This simple fact plays a fundamental role in the theory of dynamical systems,
\cite{AlpernP,Pilyugin}.
\begin{lemma}
\label{limit meas pres}
Let $\Phi_{k}:Q\to Q$, $k=1,2,\ldots$ be a Cauchy sequence of surjective homeomorphisms in the uniform metric $d$. Then
\begin{itemize}
\item[(a)] $\Phi_{k}$ converges uniformly to a homeomorphism $\Phi:Q\to Q$,
\item[(b)] $\Phi_{k}^{-1}$ converges uniformly, and the limit is equal to $\Phi^{-1}$,
\end{itemize}
If in addition the homeomorphisms $\Phi_k$ are measure preserving, then
\begin{itemize}
\item[(c)] $\Phi$ is a measure preserving homeomorphism.
\end{itemize}
\end{lemma}
\begin{proof}
Obviously $\Phi_{k}$ and $\Phi_{k}^{-1}$ are Cauchy sequences in the space of continuous mappings $C(Q,Q)$, thus they converge (uniformly)
to some $\Phi$ and $\Psi\in C(Q,Q)$, respectively.
To see that $\Psi=\Phi^{-1}$, fix a point $x\in Q$ and pass with $k$ to the limit in the equality $\Phi_{k}(\Phi_{k}^{-1}(x))=x$ to prove that $\Phi(\Psi(x))=x$.
We show that $\Psi(\Phi(x))=x$ in an analogous way. Thus $\Phi$ is a homeomorphism, and we have established (a) and (b).
To check (c), fix a compact set $A\subset Q$.  For any $\eps>0$, let $\delta>0$ be chosen in such a way that the $\delta$-tubular neighborhood
$(\Phi(A))_{\delta}$ of $\Phi(A)$ has measure not larger than $|\Phi(A)|+\eps$. Since $\Phi_{k}\xrightarrow{d} \Phi$,
we can find $k\in\bbbn$ such that $\sup_{x\in Q}|\Phi_{k}(x)-\Phi(x)|<\delta$, in particular
$\Phi_{k}(A)\subset (\Phi(A))_{\delta}$, therefore $|A|=|\Phi_{k}(A)|\leq |\Phi(A)|+\eps$. Since $\eps$ is arbitrary, we have that $|A|\leq |\Phi(A)|$.
Applying the same reasoning to the sequence $\Phi_{k}^{-1}$ and the set $\Phi(A)$ we show that
$|\Phi(A)|\leq |\Phi^{-1}(\Phi(A))|=|A|$, therefore $|\Phi(A)|=|A|$, which easily implies (c).
\end{proof}

\begin{corollary}
\label{C1}
Let $\Phi_k:Q\to Q$, $k=1,2,\ldots$ be a sequence of measure preserving diffeomorphisms that is convergent
in the uniform metric $d$ to a homeomorphism $\Phi$. Then $\Phi$ is measure preserving.
If in addition $\Phi$ is approximately differentiable a.e., then $|J_\Phi|=1$ a.e.
\end{corollary}
\begin{proof}
The fact that $\Phi$ is measure preserving follows from Lemma~\ref{limit meas pres}. In particular $\Phi$ has the Lusin property (N).
If $\Phi$ is also approximately differentiable a.e., Theorem~\ref{FedererThm} shows that the change of variables formula \eqref{E1}
is satisfied. Taking $g=\chi_{B(x,r)}$ yields
$$
\int_{B(x,r)} |J_\Phi|=|\Phi(B(x,r)| = |B(x,r)|
\quad
\text{for all $x\in Q$ and $r<\dist(x,\partial Q)$.}
$$
Hence $|J_\Phi|=1$ a.e. by the Lebesgue differentiation theorem.
\end{proof}

The main result of the paper, Theorem~\ref{main}, shows that it might happen that $J_\Phi=-1$ a.e.
even if $J_{\Phi_k}=+1$ for all $k$. This shows the complexity of the homeomorphisms that can be obtained as limits
of measure preserving diffeomorphisms.

\begin{theorem}
\label{main}
There exists an almost everywhere approximately differentiable homeomorphism $\Phi$ of the cube $Q=[0,1]^n$ onto itself, such that
\begin{itemize}
 \item[(a)] $\Phi|_{\partial Q}=\id$,
 \item[(b)] $\Phi$ is measure preserving,
 \item[(c)] $\Phi$ is a limit, in the uniform metric $d$, of measure preserving $\C^\infty$-diffeomorphisms of $Q$ that are identity on the boundary,
 \item[(d)] the approximate derivative of $\Phi$ satisfies
 \begin{equation}
 \label{E2}
\ap D\Phi(x) =
\left[
\begin{array}{ccccc}
1       &      0       &   \ldots   &   0      &     0      \\
0       &      1       &   \ldots   &   0      &     0      \\
\vdots  &   \vdots     &   \ddots   &  \vdots  &    \vdots  \\
0       &      0       &   \ldots   &   1      &     0      \\
0       &      0       &   \ldots   &   0      &    -1      \\
\end{array}
\right]
\quad
\text{a.e. in $Q$}.
 \end{equation}
 \end{itemize}
\end{theorem}
It would be much easier to prove the result without conditions (b), (c), and with the condition (d) replaced by
$\ap D\Phi<0$ a.e. (see also \cite{GH}).
However, in order to prescribe the derivative as in \eqref{E2} we had to use deep results of
Dacorogna and Moser \cite{DacorognaMoser} on the existence of diffeomorphisms with the prescribed Jacobian.

It turns out that our construction gives a lot of flexibility in prescribing the derivative of $\Phi$ and the condition \eqref{E2}
can be easily replaced by many other ones. Because of this we believe that the following conjecture is true.

\noindent
{\sc Conjecture.}
{\em
For any measurable map $T:Q\to GL(n)$ such that
\begin{equation}
\label{E3}
\int_Q |\det T|=1
\end{equation}
there exists an a.e. approximately differentiable homeomorphism $\Phi$ with the Lusin property (N) such that
$\Phi|_{\partial Q}=\id$ and $D\Phi=T$ a.e.}

Clearly, condition \eqref{E3} is necessary due to the change of variables formula, Theorem~\ref{FedererThm}, valid for such homeomorphisms $\Phi$.
Our belief in this conjecture is also supported by the results of the papers \cite{Alberti,HajlaszMirra,MoonensP}.

The paper is structured as follows. In Section~\ref{Basic} we construct a homeomorphism which gives a positive answer to the Question~1, see Lemma~\ref{base}.
In Section~\ref{Tools} we introduce tools based on the results of Dacorogna and Moser \cite{DacorognaMoser}, which allow us to construct a large class of measure preserving
diffeomorphisms. In particular, we modify Lemma~\ref{base} so that the homemorphism is measure preserving, see Lemma~\ref{base plus}. Finally, in Section~\ref{1stAttempt} we
prove Theorem~\ref{main}, and in Appendix we prove the equivalence of our definition of the approximate derivative with the classical one.

We should mention that whenever we call a homeomorphism \emph{orientation preserving}, we consider the topological definition of the orientation: a homeomorphism is orientation preserving if it has local topological degree 1 at every point of its domain. Any homeomorphism of a cube is either orientation preserving or orientation reversing (local degree -1 at every point), see e.g. \cite[II.2.2]{ReichelderferRado} for the definition of the local topological degree of a mapping and \cite[II.2.4]{ReichelderferRado} for applications to homeomorphisms. We shall not need at any point the precise definition of the local degree, using instead the following observation: if a homeomorphism of a cube is equal to identity of the boundary, it can be extended by identity to a slightly larger cube, and in the points where it coincides with identity its local degree will be 1. Therefore its local degree is 1 at all points of the original, smaller cube, proving that such a homeomorphism is orientation preserving.

Whenever a homeomorphism is differentiable at a point, one can determine its local degree as the sign of its Jacobian determinant. This is not the case any more for approximately differentiable homeomorphisms, as is shown e.g. by Theorem \ref{main}.

Notation is pretty standard. The Lebesgue measure of a set $A$ is denoted by $|A|$.
Also we will always assume that cubes have edges parallel to the coordinate directions.

\section{Basic example}
\label{Basic}
In this section we construct an a.e. approximately differentiable, sense preserving homeomorphism of the unit cube $Q$ with the Lusin property (N)
and with the Jacobian determinant equal $-1$ on a set of a positive measure.
This construction is a conceptual basis for the main result, Theorem~\ref{main}, where this example is, after necessary modifications, iterated.
The iteration, however, requires a theorem of Dacorogna and Moser on measure preserving diffeomorphisms, explained in the next section.

\begin{lemma}
\label{base}
There exists an almost everywhere approximately differentiable homeomorphism  $\Phi$ of the unit cube $Q=[0,1]^n$ onto itself with the Lusin property (N)
and a compact set $A$ in the interior of $Q$ such that
\begin{itemize}
\item $\Phi=\id$ in a neighborhood of $\partial Q$.
\item $|A|=2^{-n}$,
\item $\Phi$ is the reflection $(x_1,\ldots,x_{n-1},x_n)\mapsto (x_1,\ldots,x_{n-1},1-x_n)$ on $A$, $\Phi(A)=A$, and $\Phi$ is a $\C^\infty$-diffeomorphism outside $A$,
\item at almost all points of the set $A$
\begin{equation}
\label{E4}
\ap D\Phi(x) =
\left[
\begin{array}{ccccc}
1       &      0       &   \ldots   &   0      &     0      \\
0       &      1       &   \ldots   &   0      &     0      \\
\vdots  &   \vdots     &   \ddots   &  \vdots  &    \vdots  \\
0       &      0       &   \ldots   &   1      &     0      \\
0       &      0       &   \ldots   &   0      &    -1      \\
\end{array}
\right]\, .
\end{equation}
\end{itemize}
\end{lemma}

\begin{proof}
Let $\alpha_{0}=1$, $\alpha_k=(2^{k+1}-1)^{-1}$ for $k\in\bbbn$.  Note that, for any $k$, $2\alpha_{k}<\alpha_{k-1}$ and $\lim_{k\to\infty}2^{nk}\alpha_{k}^{n}=2^{-n}$.

Let  $Q_1^{k},\ldots,Q_{2^n}^{k}$ be closed $n$-dimensional cubes of edge-length $\alpha_{k-1}^{-1}\alpha_k<1/2$, with dyadic (i.e. with all coordinates equal $1/4$ or $3/4$) centers $q_1,\ldots,q_{2^n}$, $q_j=(q_{j,1},\ldots,q_{j,n-1},q_{j,n})$, such that $q_{2^{n-1}+j}=(q_{j,1},\ldots,q_{j,n-1},1-q_{j,n})$. This means that the first $2^{n-1}$ cubes are in the bottom layer and the
last $2^{n-1}$ are in the upper layer right above the corresponding cubes from the lower layer (see the first cube on the left in Figure~\ref{fig:cuberearr}).
\begin{figure}[h]
\begin{picture}(0,0)%
\includegraphics{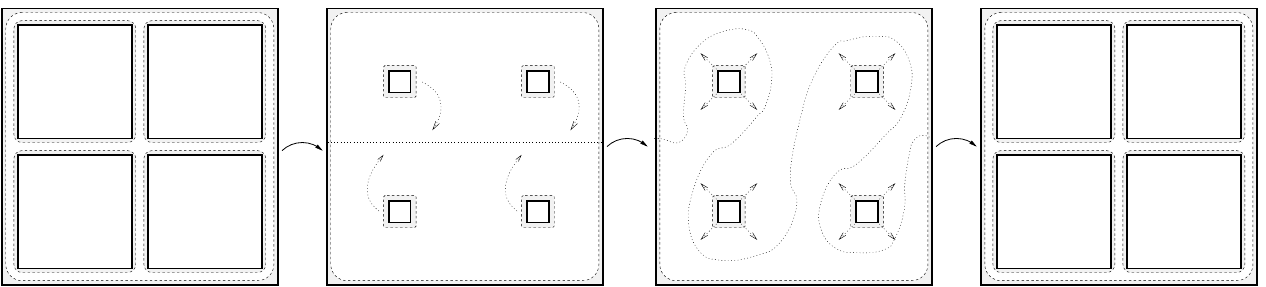}%
\end{picture}%
\setlength{\unitlength}{1026sp}%
\begingroup\makeatletter\ifx\SetFigFont\undefined%
\gdef\SetFigFont#1#2#3#4#5{%
  \reset@font\fontsize{#1}{#2pt}%
  \fontfamily{#3}\fontseries{#4}\fontshape{#5}%
  \selectfont}%
\fi\endgroup%
\begin{picture}(23241,5166)(1243,-5569)
\put(11077,-1886){\makebox(0,0)[lb]{\smash{{\SetFigFont{5}{6.0}{\rmdefault}{\bfdefault}{\updefault}{\color[rgb]{0,0,0}4}%
}}}}
\put(8527,-4286){\makebox(0,0)[lb]{\smash{{\SetFigFont{5}{6.0}{\rmdefault}{\bfdefault}{\updefault}{\color[rgb]{0,0,0}1}%
}}}}
\put(11077,-4286){\makebox(0,0)[lb]{\smash{{\SetFigFont{5}{6.0}{\rmdefault}{\bfdefault}{\updefault}{\color[rgb]{0,0,0}2}%
}}}}
\put(8527,-1886){\makebox(0,0)[lb]{\smash{{\SetFigFont{5}{6.0}{\rmdefault}{\bfdefault}{\updefault}{\color[rgb]{0,0,0}3}%
}}}}
\put(2476,-1936){\makebox(0,0)[lb]{\smash{{\SetFigFont{10}{12.0}{\rmdefault}{\bfdefault}{\updefault}{\color[rgb]{0,0,0}3}%
}}}}
\put(4876,-1936){\makebox(0,0)[lb]{\smash{{\SetFigFont{10}{12.0}{\rmdefault}{\bfdefault}{\updefault}{\color[rgb]{0,0,0}4}%
}}}}
\put(2476,-4336){\makebox(0,0)[lb]{\smash{{\SetFigFont{10}{12.0}{\rmdefault}{\bfdefault}{\updefault}{\color[rgb]{0,0,0}1}%
}}}}
\put(4876,-4336){\makebox(0,0)[lb]{\smash{{\SetFigFont{10}{12.0}{\rmdefault}{\bfdefault}{\updefault}{\color[rgb]{0,0,0}2}%
}}}}
\put(14602,-4286){\makebox(0,0)[lb]{\smash{{\SetFigFont{5}{6.0}{\rmdefault}{\bfdefault}{\updefault}{\color[rgb]{0,0,0}3}%
}}}}
\put(17152,-4286){\makebox(0,0)[lb]{\smash{{\SetFigFont{5}{6.0}{\rmdefault}{\bfdefault}{\updefault}{\color[rgb]{0,0,0}4}%
}}}}
\put(17152,-1886){\makebox(0,0)[lb]{\smash{{\SetFigFont{5}{6.0}{\rmdefault}{\bfdefault}{\updefault}{\color[rgb]{0,0,0}2}%
}}}}
\put(14602,-1886){\makebox(0,0)[lb]{\smash{{\SetFigFont{5}{6.0}{\rmdefault}{\bfdefault}{\updefault}{\color[rgb]{0,0,0}1}%
}}}}
\put(20551,-4411){\makebox(0,0)[lb]{\smash{{\SetFigFont{10}{12.0}{\rmdefault}{\bfdefault}{\updefault}{\color[rgb]{0,0,0}3}%
}}}}
\put(20551,-2011){\makebox(0,0)[lb]{\smash{{\SetFigFont{10}{12.0}{\rmdefault}{\bfdefault}{\updefault}{\color[rgb]{0,0,0}1}%
}}}}
\put(22951,-4411){\makebox(0,0)[lb]{\smash{{\SetFigFont{10}{12.0}{\rmdefault}{\bfdefault}{\updefault}{\color[rgb]{0,0,0}4}%
}}}}
\put(22951,-2011){\makebox(0,0)[lb]{\smash{{\SetFigFont{10}{12.0}{\rmdefault}{\bfdefault}{\updefault}{\color[rgb]{0,0,0}2}%
}}}}
\end{picture}%
\caption{Diffeomorphism $F_k$ rearranging the cubes.}
\label{fig:cuberearr}
\medskip
\end{figure}

Denote by
$F_{k}$ a smooth diffeomorphism exchanging $Q_{j}^{k}$ with $Q_{2^{n-1}+j}^{k}$ (the restriction of $F_{k}$ to a neighborhood of each of $Q_j$ is a translation),
that additionally is identity near a neighborhood of $\partial Q$.
A construction of such a diffeomorphism is explained on Figure~\ref{fig:cuberearr}. The only difference between diffeomorphisms $F_k$ for different values of $k$ is that they rearrange cubes $Q_j^k$ of different sizes.

Our construction is iterative. The starting point is the diffeomorphism $\Phi_{1}=F_{1}$.
This diffeomorphism rearranges cubes $Q_j^1$ of the edge-length $\alpha_0^{-1}\alpha_1=\alpha_1$ inside the unit cube $Q$.
The diffeomerphism $\Phi_2$
coincides with $\Phi_1$ on $Q\setminus\bigcup_j Q_j^1$, but in the interior of each cube $Q_j^1$, rearranged by the diffeomorphism $\Phi_1$,
$\Phi_2$ is a rescaled and translated version of the diffeomorphism $F_2$, see Figure~\ref{fig:phi2}; it rearranges $2^{2n}$ cubes of the edge-length
$\alpha_1\cdot\alpha_1^{-1}\alpha_2=\alpha_2$.
Since the diffeomorphism $F_2$ is identity near the boundary of the cube $Q$, the rescaled versions of it applied to the cubes $Q^1_j$ will be identity near boundaries of
these cubes and hence the resulting mapping $\Phi_2$ will be a smooth diffeomorphism.
The diffeomorphism $\Phi_3$ coincides with $\Phi_2$ outside the $2^{2n}$ cubes of the second generation rearranged by $\Phi_2$ and it is
a rescaled and translated version of the diffeomorphism $F_3$ inside each of the cubes rearranged by the diffeomorphism $\Phi_2$. It rearranges $2^{3n}$ cubes of the edge-length
$\alpha_2\cdot\alpha_2^{-1}\alpha_3=\alpha_3$ etc.

\begin{figure}
\begin{picture}(0,0)%
\includegraphics{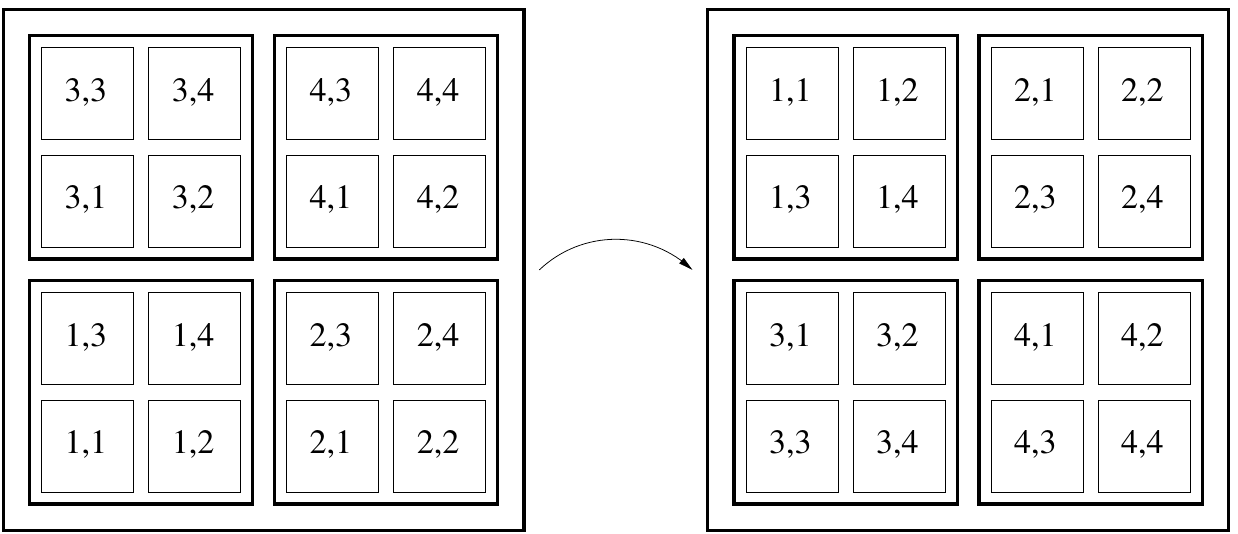}%
\end{picture}%
\setlength{\unitlength}{1934sp}%
\begingroup\makeatletter\ifx\SetFigFont\undefined%
\gdef\SetFigFont#1#2#3#4#5{%
  \reset@font\fontsize{#1}{#2pt}%
  \fontfamily{#3}\fontseries{#4}\fontshape{#5}%
  \selectfont}%
\fi\endgroup%
\begin{picture}(12066,5166)(1168,-5494)
\put(6901,-2461){\makebox(0,0)[lb]{\smash{{\SetFigFont{10}{12.0}{\familydefault}{\mddefault}{\updefault}{\color[rgb]{0,0,0}$\Phi_2$}%
}}}}
\end{picture}%
\caption{$\Phi_2$ rearranges second generation cubes}
\label{fig:phi2}
\medskip
\end{figure}

At each step we obtain a smooth diffeomorphism $\Phi_k$ of $Q$; if we denote the union of all $2^{kn}$ cubes of the $k$-th generation by $\mathcal{Q}_k$, we obtain a descending family of compact sets, with
$$
A=\bigcap_{k=1}^\infty \mathcal{Q}_k.
$$
Since $|\mathcal{Q}_k|=2^{kn}\alpha_k^n\to 2^{-n}$, $A$ is a Cantor set of measure $|A|=2^{-n}$.

The sequence $\Phi_k$ is convergent in the uniform metric $d$. Indeed, for any $m\geq k$ the diffeomorphisms $\Phi_k$, $\Phi_m$, $\Phi_k^{-1}$ and $\Phi_m^{-1}$ relocate points inside
each of the cubes of the $(k-1)$th generation. The cubes of the $(k-1)$th generation have the edge-length $\alpha_{k-1}$ and hence diameter $\sqrt{n}\alpha_{k-1}$.
Thus
$$
d(\Phi_k,\Phi_m)=\sup_{\mathcal{Q}_{k-1}}|\Phi_k-\Phi_m|+\sup_{\mathcal{Q}_{k-1}}|\Phi_k^{-1}-\Phi_m^{-1}|\leq 2\sqrt{n}\alpha_{k-1}
\quad
\text{for any $m\geq k$.}
$$
We used here the fact that $|\Phi_k-\Phi_m|+|\Phi_k^{-1}-\Phi_m^{-1}|=0$ in $Q\setminus \mathcal{Q}_{k-1}$.

This proves that $\Phi_k$ is a Cauchy sequence in the metric $d$ and thus converges to a homeomorphism $\Phi$ by Lemma~\ref{limit meas pres}.

If $x\not\in A$, then $x\not\in\mathcal{Q}_k$ for some $k$ and hence $U\cap \mathcal{Q}_k=\emptyset$ for some open neighborhood $U$ \mbox{of $x$}. Thus $\Phi_k(y)=\Phi_m(y)$ for all $m\geq k$
and $y\in U$, because
for $m>k$ the diffeomorphisms $\Phi_m$ relocate points inside $\mathcal{Q}_{m-1}\subset\mathcal{Q}_k$ only.
Accordingly $\Phi=\Phi_k$ in $U$, so $\Phi$ is a smooth diffeomorphism outside $A$.
It is also easy to see that $\Phi:Q\to Q$ acts on $A$ as the reflection $(x_1,\ldots,x_{n-1},x_n)\mapsto (x_1,\ldots,x_{n-1},1-x_n)$. In particular $\Phi(A)=A$.

The homeomorphism $\Phi$ is a diffeomorphism outside the compact set $A$ and a fixed reflection in $A$. Hence it has the Lusin property. Moreover it is differentiable
in the classical sense in $Q\setminus A$ and approximately differentiable at the density points of $A$ with the approximate derivative equal to \eqref{E4}.
The proof is complete.
\end{proof}
\begin{remark}
Note that $\Phi$ constructed above is not in $W^{1,1}(Q)$, since it is not absolutely continuous on vertical lines passing through points of $A$: it maps them into curves of infinite length (see Figure \ref{fig:noabscont}).
\end{remark}
\begin{figure}[h]

\begin{picture}(0,0)%
\includegraphics{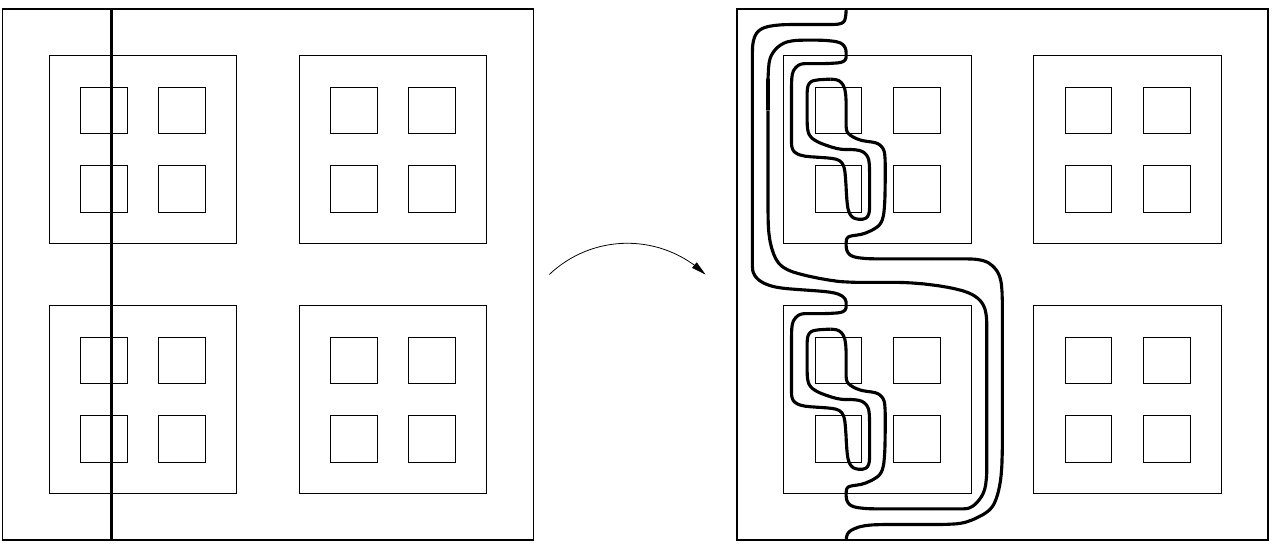}%
\end{picture}%
\setlength{\unitlength}{1973sp}%
\begingroup\makeatletter\ifx\SetFigFont\undefined%
\gdef\SetFigFont#1#2#3#4#5{%
  \reset@font\fontsize{#1}{#2pt}%
  \fontfamily{#3}\fontseries{#4}\fontshape{#5}%
  \selectfont}%
\fi\endgroup%
\begin{picture}(12194,5166)(1179,-5494)
\put(6901,-2461){\makebox(0,0)[lb]{\smash{{\SetFigFont{10}{12.0}{\familydefault}{\mddefault}{\updefault}{\color[rgb]{0,0,0}$\Phi_2$}%
}}}}
\end{picture}%
\caption{$\Phi$ maps vertical lines passing throught $C$ to curves of infinite length.}
\label{fig:noabscont}
\medskip
\end{figure}

\begin{remark}
In Section~\ref{Tools} we show that it is possible to modify the above construction so that the resulting homeomorphism is measure preserving
and it is a limit of measure preserving diffeomorphisms, see Lemma~\ref{base plus} and Corollary~\ref{000}.
\end{remark}

\section{Tools}
\label{Tools}

The following lemma is a special case of a theorem of Dacorogna and Moser (\cite[Theorem 7]{DacorognaMoser}, see also \cite[Theorem 10.11]{CsatoDacorognaKneuss}), who generalized earlier results of Moser \cite{Moser} and Banyaga~\cite{Banyaga}.
\begin{theorem}
\label{thm7DM}
Let  $\Omega$ be a bounded connected open subset of $\bbbr^n$ and let $f\in\C^{\infty}(\Omega)$ be a positive function equal 1 in a neighborhood of $\partial\Omega$ such that
$$
\int_\Omega f(x) dx=|\Omega|.
$$
Then there exists a $\C^{\infty}$ diffeomorphism $\Psi$ of $\Omega$ onto itself, that is identity on a neighborhood of $\partial\Omega$ and satisfies
$$
J_{\Psi}(x)=f(x)\quad \text{for all } x\in \Omega.$$
\end{theorem}
Although the proofs in \cite{CsatoDacorognaKneuss} and \cite{DacorognaMoser} are written only for $f$ and $\Psi\in\C^{k}(\Omega)$ for some $k\in\bbbn$, they clearly work for $f$ and $\Psi\in\C^{\infty}(\Omega)$; for a proof using Moser's \emph{flow method} with $f,~\Psi\in \C^{\infty}$ see e.g. \cite[Appendix, Lemma 2.3]{DacorognaDirect} (the first edition of the book).

As a direct corollary we will prove that every diffeomorphism between bounded domains of equal volume can be corrected to a measure preserving diffeomorphism. More precisely we have
\begin{corollary}
\label{MP}
Let $\Omega,\Omega'\subset\bbbr^n$ be bounded domains of equal volume $|\Omega|=|\Omega'|$ and let $\Psi:\Omega\to\Omega'$ be a
$\C^\infty$-diffeomorphism of $\Omega$ onto $\Omega'$ such that $J_\Psi=1$ in a neighborhood of $\partial\Omega$. Then there is another
$\C^\infty$-diffeomorphism $\Phi:\Omega\to\Omega'$ mapping $\Omega$ onto $\Omega'$ such that $\Phi=\Psi$ in a neighborhood of $\partial\Omega$ and
$J_\Phi=1$ on $\Omega$.
\end{corollary}
\begin{proof}
Let $f(x)=J_{\Psi^{-1}}(x)$. Clearly $f\in C^{\infty}$, $f>0$, $f=1$ in a neighborhood of $\partial\Omega'$, and
$$
\int_{\Omega'} f(x)\, dx = |\Psi^{-1}(\Omega')|=|\Omega|=|\Omega'|.
$$
According to Lemma~\ref{thm7DM} there is a diffeomorphism $\tilde{\Phi}:\Omega'\to\Omega'$ of class $C^{\infty}$ such that
$$
J_{\tilde{\Phi}}(x)= f(x) = J_{\Psi^{-1}}(x) = \frac{1}{J_{\Psi}(\Psi^{-1}(x))}
\quad
\text{for all $x\in\Omega$,}
$$
and $\tilde{\Phi}$ is identity in a neighborhood of $\partial\Omega'$. Let $\Phi=\tilde{\Phi}\circ\Psi$.
Clearly $\Phi=\Psi$ in a neighborhood of $\partial\Omega$ and
$$
J_{\Phi}(x) = J_{\tilde{\Phi}\circ\Psi}(x) =
J_{\tilde{\Phi}}(\Psi(x))\, J_{\Psi}(x) = \frac{J_{\Psi}(x)}{J_{\Psi}(\Psi^{-1}(\Psi(x)))} = 1.
$$
\end{proof}
The following result,  which is a direct consequence of Corollary~\ref{MP}, is a measure preserving version of a first step in the proof of Lemma~\ref{base}.
\begin{corollary}
\label{corr1}
Set $Q=[0,1]^n$. Let $Q_1,\ldots,Q_{2^n}$ be a family of disjoint, closed, $n$-dimen\-sion\-al cubes inside of $Q$, of edge length $\alpha$, $\alpha\in (0,1/2)$, with dyadic
(i.e. with all coordinates equal to $1/4$ or $3/4$) centers, such that the symmetry
$T$ with respect to the hyperplane $\{(x_1,\ldots,x_n):\,x_n=1/2\}$ maps $Q_j$ onto
$Q_{2^{n-1}+j}$ for $j=1,\ldots,2^{n-1}$.
Then there exists a $\C^\infty$-diffeomorphism $\Psi_\alpha:Q\to Q$ such that
\begin{enumerate}
\item[(a)] $\Psi_\alpha$ is identity on a neighborhood of $\partial Q$,
\item[(b)] $\Psi_\alpha$ acts as a translation on a neighborhood of each of $Q_j$, exchanging rigidly $Q_j$ with $Q_{j+2^{n-1}}$ for $j=1,\ldots,2^{n-1}$,
\item[(c)] $\Psi_\alpha$ is measure preserving.
\end{enumerate}
\end{corollary}
If in the proof of Lemma~\ref{base} we construct the diffeomorphisms $\Phi_k$ with the help of Corollary~\ref{corr1}, we obtain a slightly stronger result, Lemma~\ref{base plus};
this result will be used in Section~\ref{1stAttempt}.
\begin{lemma}
\label{base plus}
There exists a homeomorphism $\Phi$ of the unit cube $Q=[0,1]^n$ which is measure preserving and has all the properties listed in Lemma~\ref{base}.
\end{lemma}

Let $\Phi$ be the homeomorphism constructed in Lemma~\ref{base plus}. It is the identity in a neighborhood of $\partial Q$.
Moreover, $\Phi$ is defined as the limit of measure preserving diffeomorphisms $\Phi_k$ constructed with the help of Corollary~\ref{corr1}.

Let $\mathcal{Q}_k$ be the union of all $2^{nk}$ cubes of the $k$-th generation from the proof of Lemma~\ref{base} (or Lemma~\ref{base plus}).
Recall that $A=\bigcap_k\mathcal{Q}_k$ is the Cantor set on which $\Phi$ is a reflection. Observe also that $\Phi=\Phi_k$ in $Q\setminus\mathcal{Q}_k$.
Thus if $K\subset Q\setminus A$ is a compact set, we can find $k\in\bbbn$ such that $K\subset Q\setminus\mathcal{Q}_k$ and hence
$\Phi=\Phi_k$ on $K$.
Since the diffeomorphisms $\Phi_k$ converge to $\Phi$ in the uniform metric $d$, we have
\begin{corollary}
\label{000}
For any $\eps>0$ and any compact set $K\subset Q\setminus A$ there is a measure preserving $\C^\infty$-diffeomorphism $\tilde{\Phi}$ of $Q$ onto itself such that
$\tilde{\Phi}=\id$ in a neighborhood of $\partial Q$, $\tilde{\Phi}=\Phi$ on $K$ and $d(\Phi,\tilde{\Phi})<\eps$.
\end{corollary}

Consider the situation presented on Figure 4. The complement of a finite number of disjoint  cubes in a ball is obviously smoothly diffeomorphic to the analogous complement of the same number of disjoint cubes in an ellipsoid; by a procedure similar to the one on Figure 1 one can construct a diffeomorphism that rearranges these cubes, together with their small neighborhoods, in a prescribed way, and which is an affine mapping near the boundary of the ball. Then, Corollary 3.2 immediately yields the following result, which will be used in the proof of Theorem 1.4.

\begin{corollary}
\label{corr2}
Let $\mathcal{Q}=\{Q_1,\ldots,Q_k\}$ be a finite family of identical, disjoint, closed cubes of edge length $q$, in an $n$-dimensional ball $B$.
Let $E$ be an $n$-dimensional ellipsoid, $|E|=|B|$, and let $\tilde{\mathcal{Q}}=\{\tilde{Q}_1,\ldots,\tilde{Q}_k\}$
be a family of identical, disjoint, closed cubes
of edge $q$, in $E$. Then there exists a measure preserving $\C^\infty$-diffeomorphism $\Phi:\bar{B}\to \bar{E}$, linear on the boundary and its neighborhood,
such that $\Phi(Q_j)=\tilde{Q}_j$ for $j=1,2,\ldots,k$ and $\Phi$ is a translation in a neighborhood of each cube $Q_j$.
\end{corollary}
\begin{figure}[h]
\begin{picture}(0,0)%
\includegraphics{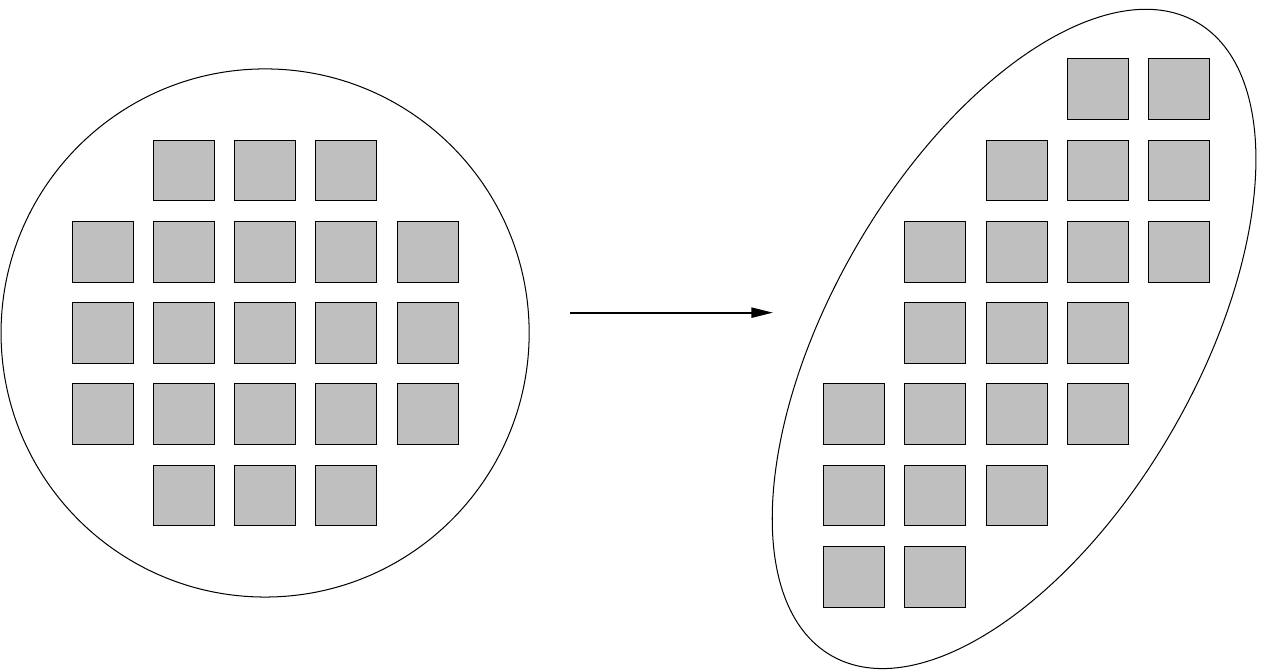}%
\end{picture}%
\setlength{\unitlength}{2565sp}%
\begingroup\makeatletter\ifx\SetFigFont\undefined%
\gdef\SetFigFont#1#2#3#4#5{%
  \reset@font\fontsize{#1}{#2pt}%
  \fontfamily{#3}\fontseries{#4}\fontshape{#5}%
  \selectfont}%
\fi\endgroup%
\begin{picture}(9283,4882)(1193,-7796)
\put(2476,-4186){\makebox(0,0)[lb]{\smash{{\SetFigFont{10}{12.0}{\familydefault}{\mddefault}{\updefault}{\color[rgb]{0,0,0}1}%
}}}}
\put(3001,-5386){\makebox(0,0)[lb]{\smash{{\SetFigFont{10}{12.0}{\familydefault}{\mddefault}{\updefault}{\color[rgb]{0,0,0}11}%
}}}}
\put(1801,-5986){\makebox(0,0)[lb]{\smash{{\SetFigFont{10}{12.0}{\familydefault}{\mddefault}{\updefault}{\color[rgb]{0,0,0}14}%
}}}}
\put(3076,-4186){\makebox(0,0)[lb]{\smash{{\SetFigFont{10}{12.0}{\familydefault}{\mddefault}{\updefault}{\color[rgb]{0,0,0}2}%
}}}}
\put(3676,-4186){\makebox(0,0)[lb]{\smash{{\SetFigFont{10}{12.0}{\familydefault}{\mddefault}{\updefault}{\color[rgb]{0,0,0}3}%
}}}}
\put(1876,-4786){\makebox(0,0)[lb]{\smash{{\SetFigFont{10}{12.0}{\familydefault}{\mddefault}{\updefault}{\color[rgb]{0,0,0}4}%
}}}}
\put(2476,-4786){\makebox(0,0)[lb]{\smash{{\SetFigFont{10}{12.0}{\familydefault}{\mddefault}{\updefault}{\color[rgb]{0,0,0}5}%
}}}}
\put(3076,-4786){\makebox(0,0)[lb]{\smash{{\SetFigFont{10}{12.0}{\familydefault}{\mddefault}{\updefault}{\color[rgb]{0,0,0}6}%
}}}}
\put(3676,-4786){\makebox(0,0)[lb]{\smash{{\SetFigFont{10}{12.0}{\familydefault}{\mddefault}{\updefault}{\color[rgb]{0,0,0}7}%
}}}}
\put(4276,-4786){\makebox(0,0)[lb]{\smash{{\SetFigFont{10}{12.0}{\familydefault}{\mddefault}{\updefault}{\color[rgb]{0,0,0}8}%
}}}}
\put(1876,-5386){\makebox(0,0)[lb]{\smash{{\SetFigFont{10}{12.0}{\familydefault}{\mddefault}{\updefault}{\color[rgb]{0,0,0}9}%
}}}}
\put(2401,-5386){\makebox(0,0)[lb]{\smash{{\SetFigFont{10}{12.0}{\familydefault}{\mddefault}{\updefault}{\color[rgb]{0,0,0}10}%
}}}}
\put(3601,-5386){\makebox(0,0)[lb]{\smash{{\SetFigFont{10}{12.0}{\familydefault}{\mddefault}{\updefault}{\color[rgb]{0,0,0}12}%
}}}}
\put(4201,-5386){\makebox(0,0)[lb]{\smash{{\SetFigFont{10}{12.0}{\familydefault}{\mddefault}{\updefault}{\color[rgb]{0,0,0}13}%
}}}}
\put(2401,-5986){\makebox(0,0)[lb]{\smash{{\SetFigFont{10}{12.0}{\familydefault}{\mddefault}{\updefault}{\color[rgb]{0,0,0}15}%
}}}}
\put(3001,-5986){\makebox(0,0)[lb]{\smash{{\SetFigFont{10}{12.0}{\familydefault}{\mddefault}{\updefault}{\color[rgb]{0,0,0}16}%
}}}}
\put(3601,-5986){\makebox(0,0)[lb]{\smash{{\SetFigFont{10}{12.0}{\familydefault}{\mddefault}{\updefault}{\color[rgb]{0,0,0}17}%
}}}}
\put(4201,-5986){\makebox(0,0)[lb]{\smash{{\SetFigFont{10}{12.0}{\familydefault}{\mddefault}{\updefault}{\color[rgb]{0,0,0}18}%
}}}}
\put(2401,-6586){\makebox(0,0)[lb]{\smash{{\SetFigFont{10}{12.0}{\familydefault}{\mddefault}{\updefault}{\color[rgb]{0,0,0}19}%
}}}}
\put(3001,-6586){\makebox(0,0)[lb]{\smash{{\SetFigFont{10}{12.0}{\familydefault}{\mddefault}{\updefault}{\color[rgb]{0,0,0}20}%
}}}}
\put(3601,-6586){\makebox(0,0)[lb]{\smash{{\SetFigFont{10}{12.0}{\familydefault}{\mddefault}{\updefault}{\color[rgb]{0,0,0}21}%
}}}}
\put(9226,-3586){\makebox(0,0)[lb]{\smash{{\SetFigFont{10}{12.0}{\familydefault}{\mddefault}{\updefault}{\color[rgb]{0,0,0}1}%
}}}}
\put(9826,-3586){\makebox(0,0)[lb]{\smash{{\SetFigFont{10}{12.0}{\familydefault}{\mddefault}{\updefault}{\color[rgb]{0,0,0}2}%
}}}}
\put(8626,-4186){\makebox(0,0)[lb]{\smash{{\SetFigFont{10}{12.0}{\familydefault}{\mddefault}{\updefault}{\color[rgb]{0,0,0}3}%
}}}}
\put(9226,-4186){\makebox(0,0)[lb]{\smash{{\SetFigFont{10}{12.0}{\familydefault}{\mddefault}{\updefault}{\color[rgb]{0,0,0}4}%
}}}}
\put(9826,-4186){\makebox(0,0)[lb]{\smash{{\SetFigFont{10}{12.0}{\familydefault}{\mddefault}{\updefault}{\color[rgb]{0,0,0}5}%
}}}}
\put(8026,-4786){\makebox(0,0)[lb]{\smash{{\SetFigFont{10}{12.0}{\familydefault}{\mddefault}{\updefault}{\color[rgb]{0,0,0}6}%
}}}}
\put(8626,-4786){\makebox(0,0)[lb]{\smash{{\SetFigFont{10}{12.0}{\familydefault}{\mddefault}{\updefault}{\color[rgb]{0,0,0}7}%
}}}}
\put(9226,-4786){\makebox(0,0)[lb]{\smash{{\SetFigFont{10}{12.0}{\familydefault}{\mddefault}{\updefault}{\color[rgb]{0,0,0}8}%
}}}}
\put(9826,-4786){\makebox(0,0)[lb]{\smash{{\SetFigFont{10}{12.0}{\familydefault}{\mddefault}{\updefault}{\color[rgb]{0,0,0}9}%
}}}}
\put(7951,-5386){\makebox(0,0)[lb]{\smash{{\SetFigFont{10}{12.0}{\familydefault}{\mddefault}{\updefault}{\color[rgb]{0,0,0}10}%
}}}}
\put(8551,-5386){\makebox(0,0)[lb]{\smash{{\SetFigFont{10}{12.0}{\familydefault}{\mddefault}{\updefault}{\color[rgb]{0,0,0}11}%
}}}}
\put(9151,-5386){\makebox(0,0)[lb]{\smash{{\SetFigFont{10}{12.0}{\familydefault}{\mddefault}{\updefault}{\color[rgb]{0,0,0}12}%
}}}}
\put(7351,-5986){\makebox(0,0)[lb]{\smash{{\SetFigFont{10}{12.0}{\familydefault}{\mddefault}{\updefault}{\color[rgb]{0,0,0}13}%
}}}}
\put(7951,-5986){\makebox(0,0)[lb]{\smash{{\SetFigFont{10}{12.0}{\familydefault}{\mddefault}{\updefault}{\color[rgb]{0,0,0}14}%
}}}}
\put(8551,-5986){\makebox(0,0)[lb]{\smash{{\SetFigFont{10}{12.0}{\familydefault}{\mddefault}{\updefault}{\color[rgb]{0,0,0}15}%
}}}}
\put(9151,-5986){\makebox(0,0)[lb]{\smash{{\SetFigFont{10}{12.0}{\familydefault}{\mddefault}{\updefault}{\color[rgb]{0,0,0}16}%
}}}}
\put(7351,-6586){\makebox(0,0)[lb]{\smash{{\SetFigFont{10}{12.0}{\familydefault}{\mddefault}{\updefault}{\color[rgb]{0,0,0}17}%
}}}}
\put(7951,-6586){\makebox(0,0)[lb]{\smash{{\SetFigFont{10}{12.0}{\familydefault}{\mddefault}{\updefault}{\color[rgb]{0,0,0}18}%
}}}}
\put(8551,-6586){\makebox(0,0)[lb]{\smash{{\SetFigFont{10}{12.0}{\familydefault}{\mddefault}{\updefault}{\color[rgb]{0,0,0}19}%
}}}}
\put(7351,-7186){\makebox(0,0)[lb]{\smash{{\SetFigFont{10}{12.0}{\familydefault}{\mddefault}{\updefault}{\color[rgb]{0,0,0}20}%
}}}}
\put(7951,-7186){\makebox(0,0)[lb]{\smash{{\SetFigFont{10}{12.0}{\familydefault}{\mddefault}{\updefault}{\color[rgb]{0,0,0}21}%
}}}}
\end{picture}
\caption{Rearrangement of a sufficiently large family of cubes between a ball and an ellipsoid.}
\label{circtoellipse3}
\end{figure}

In order to be able to apply Corollary~\ref{corr2} in the iterative procedure used in the proof of Theorem~\ref{main} we need to be able to correct measure preserving
diffeomorphisms in a way that they remain measure preserving and they map
a certain finite but large family of disjoint small balls onto ellipsoids of equal volume. This can be done thanks to Lemma~\ref{goodapprox} whose proof is yet another application of
Corollary~\ref{MP}. The lemma shows that if \mbox{$\Phi:B(x_o,r)\to\bbbr^{n}$,} $B(x_{o},r)\subset \bbbr^{n}$, is a measure preserving diffeomorphism of class $C^{\infty}$ that is sufficiently
well approximated by its tangent map $T(x)=\Phi(x_{o})+D\Phi(x_{o})(x-x_{o})$, then there exists another measure preserving diffeomorphism $\tilde{\Phi}:B(x_{o},r)\to\bbbr^{n}$,
that coincides with $\Phi$ near $\partial B(x_{o},r)$ and equals $T$ on $B(x_{o},r/2)$. The actual statement is quite technical and it requires some notation.

For a diffeomorphism $\Phi:\Omega\to\bbbr^{n}$, $\Omega\subset \bbbr^{n}$, of class $C^{\infty}$ we define
\begin{eqnarray*}
\|D\Phi\|_{\Omega}&=&\sup_{x\in\Omega}\|D\Phi(x)\| =\sup_{x\in\Omega} \sup_{|\xi|=1}\left|D\Phi(x)\xi\right|,\\
\|(D\Phi)^{-1}\|_{\Omega}&=&\sup_{x\in\Omega}\|(D\Phi)^{-1}(x)\|=\sup_{x\in\Omega}\sup_{|\xi|=1}\left| (D\Phi(x))^{-1}\xi\right|\\
\|D^{2}\Phi\|_{\Omega}&=&\sup_{x\in\Omega}\|D^{2}\Phi(x)\|=\sup_{x\in\Omega}\sup_{|\xi|=|\eta|=1}\left|\sum_{i,j=1}^{n} \frac{\partial^{2}\Phi}{\partial x_{i}\partial x_{j}}(x)\xi_{i}\eta_{j}\right|.
\end{eqnarray*}
If $B(x_{o},r)\subset\Omega$ and  $x,y\in B(x_{o},r)$, then
$$\left| \Phi(x)-\Phi(y)\right|\leq \|D\Phi\|_{\Omega}|x-y|\quad \text{ and }\quad
\left| \Phi(x)-T(x) \right|\leq \|D^{2}\Phi\|_{\Omega}|x-x_{o}|^{2}.$$ 
The last inequality follows from Taylor's formula:
$$
\Phi(x_{o}+h)-\Phi(x_{o})-D\Phi(x_{o})h=\sum_{i,j=1}^{n}h_{i}h_{j}\int_{0}^{1}(1-t)\frac{\partial^{2}\Phi}{\partial x_{i}\partial x_{j}}(x_{o}+th)\, dt.
$$
We start with a simple topological observation.
\begin{lemma}\label{lemma:top}
Assume $B$ is a closed ball in $\bbbr^n$ centered at $x$, and $F,G\colon B\to\bbbr^n$ are two homeomorphisms (onto their respective images). If $F(x)=G(x)$ and $F(\partial B)=G(\partial B)$, then $F(B)=G(B)$.
\end{lemma}
\begin{proof}
Assume there is a point $y\in B$ such that $F(y)\not\in G(B)$. Then the image $F([x,y])$ of the line segment $[x,y]$ intersects $G(\partial B)=F(\partial B)$, thus there exists $z\in (x,y)$ and $w\in\partial B$ such that $F(z)=F(w)$, which implies that $z=w$. However, $z$ is an interior point of $B$, which yields a contradiction.
\end{proof}
\begin{lemma}
\label{goodapprox}
Let $\Phi:\Omega\to\bbbr^{n}$, $\Omega\subset\bbbr^{n}$, be a measure preserving diffeomorphism of class $C^{\infty}$ such that
$$
M=\|D\Phi\|_{\Omega}+\|(D\Phi)^{-1}\|_{\Omega}+\|D^{2}\Phi\|_{\Omega}<\infty.
$$
Let $B=B(x_{o},r) \Subset \Omega$ and $D=\overline{B}(x_{o},r/2)$.
If
$$
r<\left(10 (M+1)^{2}2^{\ell}\right)^{-1}\quad\text{ for some }\ell\in \bbbn,
$$
then
\begin{itemize}
\item[(a)] $\diam \Phi(B)<2^{-\ell}$,
\item[(b)] $T(D)\subset \Phi(B)$, where $T(x)=\Phi(x_{o})+D\Phi(x_{o})(x-x_{o})$,
\item[(c)] there is a measure preserving $\C^\infty$-diffeomorphism $\tilde{\Phi}$ which coincides with $\Phi$ on
$\Omega\setminus B(x_{o},r-\eps)$ and coincides with $T$ on $B(x_{o},\frac{r}{2}+\eps)$ for some $\eps=\eps(r)>0$ .
\end{itemize}
\end{lemma}
\begin{proof}
By translating the domain we may assume that $x_o=0$. That will slightly simplify the notation.
For any $x,y\in \overline{B}$ we have
$$
\left| \Phi(x)-\Phi(y) \right|\leq \|D\Phi\|_{\Omega}|x-y|\leq M\cdot 2r<2^{-\ell},
$$
which implies (a).

To establish (c), fix a smooth, non-decreasing function $\phi:\bbbr\to[0,1]$, $\|\phi'\|_{\infty}<9$, such that
$\phi(t)=0$ for $t\leq 3/5$ and $\phi(t)=1$ for $t\geq 4/5$, and define
$$
G(x)=T(x)+\phi(|x|/r)(\Phi(x)-T(x)):=T(x)+L(x).
$$
Observe that $G$ coincides with $\Phi$ near $\partial B(0,r)$ and coincides with $T$ in a neighborhood of~$D$. We will prove that $G$ is a diffeomorphism.

For $x\in\overline{B}$ we have
\begin{equation*}
\begin{split}
\|DL(x)\|&\leq \frac{1}{r}\phi'\left(\frac{|x|}{r}\right)| \Phi(x)-T(x)|+\|D\Phi(x)-D\Phi(0)\|\\
&\leq \frac{9}{r}\|D^{2}\Phi\|_{\Omega}|x|^{2}+\|D^{2}\Phi\|_{\Omega}|x|\\
&\leq 10Mr.
\end{split}
\end{equation*}
Hence, for $x,y\in \overline{B}$,
\begin{equation}\label{eq:inject}
\begin{split}
|G(x)-G(y)|&\geq |T(x)-T(y)|-|L(x)-L(y)|\\
&\geq |D\Phi(0)(x-y)|-\sup_{z\in\overline{B}}\|DL(z)\| |x-y|\\
&\geq |D\Phi(0)(x-y)|-10Mr |x-y|\\
&\geq |D\Phi(0)(x-y)|-10M r \|(D\Phi)^{-1}\|_{\Omega}|D\Phi(0)(x-y)|\\
&\geq \left(1-10M^{2}r\right) |D\Phi(0)(x-y)|\\
&\geq \frac{1}{2}|D\Phi(0)(x-y)|.
\end{split}
\end{equation}
This estimate shows that $G$ is injective, which, by compactness of $\overline{B}$, proves that $G$ is a homeomorphism on $\overline{B}$. 
As we noted before, $G$ coincides with $\Phi$ in the neighborhood of $\partial B$, also $G(x_o)=\Phi(x_o)$. 
Applying Lemma \ref{lemma:top} to these two homeomorphisms we obtain that $G(\overline{B})=\Phi(\overline{B})$, 
which in turn immediately implies (b), since $G=T$ in the neighborhood of $D$.

In fact, the estimate \eqref{eq:inject} implies more -- that $G$ is a diffeomorphism. To prove that, it suffices to show non-degeneracy of $DG$ at points of $B$ (in all other points $G$ coincides with $\Phi$); assume $x\in B$. For an arbitrary $v\in \bbbr^n$ and sufficiently small $\tau>0$ we have $y=x+\tau v\in B$. Then, by \eqref{eq:inject},
\begin{equation*}\begin{split}
\frac{1}{2}|D\Phi(0)(\tau v)|= \frac{1}{2}|D\Phi(0)(x-y)|&\leq |G(x)-G(y)|\\
&=|DG(x)(y-x)+o(|x-y|)|=|\tau DG(x)v +o(\tau)|.
\end{split}\end{equation*}
Dividing both sides by $\tau$ and passing with $\tau$ to $0$ gives $|DG(x)v|\geq \frac{1}{2}|D\Phi(0)(v)|$, and since $D\Phi(0)$ is non-degenerate, so is $DG(x)$.

Note that $G(B)=\Phi(B)$ has the same measure as $B$; also the tangent mapping $T$ is measure preserving,
therefore the measure of $G(D)=T(D)$ is the same as that of $D$.
Thus Corollary~\ref{MP} allows us to change the diffeomorphism $G$ to a measure preserving diffeomorphism $\tilde{\Phi}$, which
satisfies the condition (c).
\end{proof}
A similar argument of interpolating between an affine mapping on a smaller ball and the original diffeomorphism outside a larger one has been used before (and also in the context of Dacorogna-Moser's Theorem) in the paper of M\"uller  and \v{S}ver\'ak \cite{MuellerSverak}.

\section{Proof of Theorem~\ref{main}}
\label{1stAttempt}
The proof of Theorem~\ref{main} is divided into two parts.

In the first part (Section~\ref{first}),
we present a construction of an almost everywhere approximately differentiable homeomorphism $F$ of
$Q$ onto $Q$ that satisfies all the conditions listed in Theorem \ref{main}, except it is not given as a uniform limit of diffeomorphisms.
(We denote the constructed diffeomorphism by $F$ instead of $\Phi$, because we use the homeomorphism $\Phi$ from Lemma~\ref{base plus} and we want to avoid confusion).
The homeomorphism $F$ is constructed as a limit of homeomorphisms $F_k$ in the uniform metric~$d$.

In the second part (Section~\ref{second}), we
prove that for any $k$ there is a measure and orientation preserving $\C^\infty$-diffeomorphism $\Xi_k$, such that $\Xi_k|_{\partial Q}=\id$
and $d(\Xi_k,F_k)<2^{-k}$. Hence $\Xi_k\to F$ and this completes the proof of Theorem~\ref{main}.

\subsection{Constructing $F_k$ and $F$}
\label{first}\mbox{}\\
We take, as the starting point, the homeomorphism $\Phi$  constructed in Lemma~\ref{base plus}: \mbox{$F_1=\Phi$}. We set $C_1=A$.

In the inductive step we assume that we have a measure preserving, almost everywhere approximately differentiable homeomorphism
$F_k:Q\to Q$ and a compact set $C_k$ in the interior of $Q$ such that
\begin{itemize}
\item $C_k$ has positive measure,
\item $F_k$ is a measure preserving homeomorphism which is a
$\C^\infty$-diffeomorphism outside~$C_k$, $F_k|_{\partial Q}=\id$,
\item $F_k$ is approximately differentiable at almost all points of $C_k$ and
\begin{equation*}
\ap DF_{k}(x) = \mathfrak{R}:=
\left[
\begin{array}{ccccc}
1       &      0       &   \ldots   &   0      &     0      \\
0       &      1       &   \ldots   &   0      &     0      \\
\vdots  &   \vdots     &   \ddots   &  \vdots  &    \vdots  \\
0       &      0       &   \ldots   &   1      &     0      \\
0       &      0       &   \ldots   &   0      &    -1      \\
\end{array}
\right]
\qquad
\text{a.e. in $C_k$.}
 \end{equation*}
\end{itemize}

We construct $F_{k+1}$ by modifying $F_k$ in such a way that $F_{k+1}=F_k$ in a neighborhood of~$C_k$ and that there is a compact set $E_{k+1}$
in the interior of $Q\setminus C_k$, such that \mbox{$|E_{k+1}|>2^{-(2n+3)}|Q\setminus C_k|$}
and $F_{k+1}$ has the same properties as the properties of $F_k$ listed above, with $C_{k+1}=C_k\cup E_{k+1}$.
Moreover, the uniform distance between $F_k$ and $F_{k+1}$ satisfies $d(F_k,F_{k+1})<2^{-k+1}$. This guarantees that $F_k$ is a Cauchy sequence
in the metric $d$ and hence it converges to a measure preserving homeomorphism $F$, $F|_{\partial Q}=\id$, by Lemma~\ref{limit meas pres}.
Since $|Q\setminus C_{k+1}|<(1-2^{-(2n+3)})|Q\setminus C_k|$, it follows that $|\bigcup_k C_k|=|Q|$.
As $F=F_k$ on $C_k$, we have that $\ap DF|_{C_k}=\ap DF_k|_{C_k} = \mathfrak{R}$ at almost all points of $C_k$ and hence $\ap DF=\mathfrak{R}$ a.e. in $Q$.

The construction of $F_{k+1}$ will be described in three subsequent stages.

\noindent
{\bf Step 1. Linearization.}
Let $\Omega_k\Subset Q\setminus C_k$ be an open set such that $|\Omega_k|>\frac{1}{2}|Q\setminus C_k|$ and let
$$
M=\Vert D F_k\Vert_{\overline{\Omega}_k} +
\Vert (D F_k)^{-1}\Vert_{\overline{\Omega}_k} +
\Vert D^2 F_k\Vert_{\overline{\Omega}_k}.
$$
Clearly $M<\infty$.

Let $\mathcal{B}=\{ B_{ki}\}_{i=1}^{N_k}$
be a finite family of balls whose closures are disjoint sets inside $\Omega_k$, each of radius less than $(10(M+1)^2 2^k)^{-1}$ and such that
$|\bigcup_{i=1}^{N_k} B_{ki}|>4^{-1}|Q\setminus C_k|$.

Note that $\diam B_{ki}<2^{-k}$ and $\diam F_k(B_{ki})<2^{-k}$ by Lemma~\ref{goodapprox}(a).

We construct $F'_k$ by modifying $F_k$ in each ball $B_{ki}$ according to Lemma~\ref{goodapprox}.
This way $F'_k$ coincides with $F_k$ in a neighborhood of $C_k$ and $F'_k$ is a measure preserving diffeomorphism in $Q\setminus C_k$.
In particular, it is a measure preserving diffeomorphism in $\Omega_k$.

Note that the diffeomorphism $F'_k$ is linear in each ball $D_{ki}=\overline{B}(x_o,r/2)$, where $B(x_o,r)=B_{ki}\in\mathcal{B}$. We also have
$|\bigcup_{i=1}^{N_k} D_{ki}|=2^{-n}|\bigcup_{i=1}^{N_k} B_{kn}|>4^{-1}2^{-n}|Q\setminus C_k|$.

\noindent
{\bf Step 2. Rearranging small cubes.}
In this step we modify $F'_k$ in each ball $D_{ki}$, \mbox{$i=1,2,\ldots,N_k$}, constructed in Step~1, according to Corollary~\ref{corr2}.
The resulting mapping is denoted by $F''_k$.

The ball $D_{ki}$ and the ellipsoid $F'_k(D_{ki})$ have the same measure. Therefore, we can find a finite family of small, identical, disjoint cubes $\mathcal{Q}_{ki}=\{ Q_{ki}^j\}_j$ in $D_{ki}$ such that the cubes in $\mathcal{Q}_{ki}$ cover at least $1/2$ of the measure of $D_{ki}$ -- and another family, of the same number of disjoint cubes isometric to the ones in $\mathcal{Q}_{ki}$, but this time in the ellipsoid $F'_k(D_{ki})$ (one can draw a sufficiently dense grid in $Q$, then choose the cubes defined by the grid, inscribed into $D_{ki}$ and $F'_k(D_{ki})$ and slightly shrink them to make them disjoint, finally elliminate some of them if the numbers of cubes in both sets do not match).

The map $F_k''$ rearranges the cubes in $D_{ki}$ by translations to the cubes in $F'_k(D_{ki})$, as depicted in Figure \ref{circtoellipse3} and described in Corollary~\ref{corr2}.
Note that the total measure of cubes rearranged by $F''_k$ satisfies
$|\bigcup_{ij} Q_{ki}^j|>8^{-1}2^{-n}|Q\setminus C_k|$.

\noindent
{\bf Step 3. Gluing in the basic construction.}
In this last stage we modify $F''_k$ within each of the cubes $Q_{ki}^j\subset D_{ki}$
constructed in Step~2.

Recall that $F''_k$ acts on $Q_{ki}^j\subset D_{ki}$ as a translation.
We modify it by superposing it inside $Q_{ki}^j$ with a properly scaled version of $\Phi$ from Lemma~\ref{base plus}. The resulting map is denoted by $F_{k+1}$.

In each cube $Q_{ki}^j$ there is a Cantor set $A_{ki}^j$ of measure $2^{-n}|Q_{ki}^j|$, on which $F_{k+1}$ acts as a
reflection + translation
(the reflection is inherited from Lemma~\ref{base plus} and the translation is inherited from $F_k''$)
so  $\ap DF_{k+1}=\mathfrak{R}$ a.e. in $A_{ki}^j$.
Note that the rescaled version of $\Phi$ inserted into $Q_{ki}^j$ is a smooth measure preserving diffeomorphism in
$Q_{ki}^j\setminus A_{ki}^j$. Clearly,
$$
\Big|\bigcup_{ij} A_{ki}^j\Big|=2^{-n}\Big|\bigcup_{ij} Q_{ki}^j\Big|>8^{-1}2^{-2n}|Q\setminus C_k|=2^{-(2n+3)}|Q\setminus C_k|.
$$

We set $E_{k+1}=\bigcup_{ij}A_{ki}^j$ and $C_{k+1}=C_k\cup E_{k+1}$. Obviously, $\ap DF_{k+1}=\mathfrak{R}$ a.e. in $E_{k+1}$,
and hence a.e. in $C_{k+1}$. Moreover, $F_{k+1}$ is a smooth measure preserving diffeomorphism on $Q\setminus C_{k+1}$ and $F_{k+1}|_{\partial Q}=\id$.

It follows from our construction that $F_{k+1}=F_k$ in $Q\setminus\bigcup_{i=1}^{N_k} B_{ki}$ and $F_k(B_{ki})=F_{k+1}(B_{ki})$.
Actually,
\begin{equation}
\label{bimbala}
F_{k+1}=F_k
\quad
\text{near the boundary of $W_k=\bigcup_{i=1}^{N_k} \overline{B}_{ki}$.}
\end{equation}

Since $\diam F_k(B_{ki})<2^{-k}$ for $i=1,2,\ldots,N_k$, we have that
$$
\sup_{x\in Q} |F_k(x)-F_{k+1}(x)|<2^{-k}.
$$
Also $F_k^{-1}=F_{k+1}^{-1}$ in $Q\setminus\bigcup_{i=1}^{N_k} F_k(B_{ki})$. Since
$$
F_k^{-1}(F_k(B_{ki}))=F_{k+1}^{-1}(F_k(B_{ki}))=B_{ki}
\quad
\text{and}
\quad
\diam B_{ki}<2^{-k},
$$
it follows that
$$
\sup_{x\in Q} |F_k^{-1}(x)-F_{k+1}^{-1}(x)|<2^{-k}.
$$
Thus $d(F_k,F_{k+1})<2^{-k+1}$, so $F_k$ is a Cauchy sequence in the metric $d$
and hence it converges to a measure preserving homeomorphism $F$ by Lemma~\ref{limit meas pres}.

This completes the first part of the proof of Theorem~\ref{main}.
We have not proven yet that the homeomorphism $F$ can be approximated by measure preserving diffeomorphism.

The homeomorphism $F$ has a strange property: the union of the Cantor sets $A\cup\bigcup_{ijk} A_{ki}^j=\bigcup_k C_k$ is a subset of $Q$ of full measure.
The homeomorphism $F$ is a reflection on $A$, a reflection plus a translation on each of $A_{ki}^j$, and yet, $F$ is an orientation preserving homeomorphism of the cube $Q$.

\subsection{Constructing $\Xi_k$ and finishing the proof}
\label{second}\mbox{}\\
In order to complete the proof, it suffices to show that for any $k$ there is a measure preserving $\C^\infty$-diffeomorphism $\Xi_k$ such that
$\Xi_k=\id$ on $\partial Q$ and $d(F_k,\Xi_k)<2^{-k}$. To this end, it suffices to prove the next lemma, because the diffeomorphism
$\Xi_k=\Xi_{kk}$ will have all required properties.

\begin{lemma}
\label{Le}
For every $k\in\bbbn$ and $\ell=1,2,\ldots,k$ there are measure preserving $\C^\infty$\nobreakdash -\hspace{0pt}diffeomorphisms
$\Xi_{k\ell}$ such that
\begin{equation}
\label{Eq1}
\Xi_{k\ell} = \id
\quad
\text{on}
\quad
\partial Q,
\end{equation}
\begin{equation}
\label{Eq2}
\Xi_{k\ell} = F_\ell
\quad
\text{on}
\quad
\overline{\Omega}_\ell\cup\ldots\cup\overline{\Omega}_k,
\end{equation}
\begin{equation}
\label{Eq3}
d(\Xi_{k\ell},F_\ell)< 2^{-2k+\ell}.
\end{equation}
\end{lemma}
\begin{remark}
Let us comment on condition \eqref{Eq2}. We require $\Xi_{k\ell}$ to be a diffeomorphism, but $F_\ell$ is only a homeomorphism.
There is, however, no contradiction here, because $F_\ell$ is a diffeomorphism outside $C_\ell$, so, in particular,
it is a diffeomorphism in $\overline{\Omega}_{\ell}\cup\ldots\cup\overline{\Omega}_k$. Indeed, for $j\geq \ell$,
$\overline{\Omega}_j\subset Q\setminus C_j\subset Q\setminus C_{\ell}$, so
$\overline{\Omega}_{\ell}\cup\ldots\cup\overline{\Omega}_k\subset Q\setminus C_{\ell}$.
\end{remark}
\begin{proof}
Recall that $F_1=\Phi$ is a homeomorphism from Lemma~\ref{base plus}. Let
$\Xi_{k1}=\tilde{\Phi}$ be a diffeomorphism from Corollary~\ref{000} such that
$$
\Xi_{k1} = \Phi = F_1
\quad
\text{on}
\quad
\overline{\Omega}_1\cup\ldots\cup\overline{\Omega}_k
$$
and
$$
d(\Xi_{k1},F_1)= d(\tilde{\Phi},\Phi)<2^{-2k+1}.
$$
We used here the fact that
$$
\overline{\Omega}_1\cup\ldots\cup\overline{\Omega}_k\subset Q\setminus C_1=Q\setminus A.
$$
Now suppose that we already constructed a measure preserving diffeomorphism $\Xi_{k\ell}$ for
some $1\leq\ell\leq k-1$ that has properties
\eqref{Eq1}, \eqref{Eq2} and \eqref{Eq3}.

We need to construct a measure preserving diffeomorphism $\Xi_{k,\ell+1}$ that satisfies
\begin{equation}
\label{Eq4}
\Xi_{k,\ell+1}=\id
\quad
\text{on}
\quad
\partial Q,
\end{equation}
\begin{equation}
\label{Eq5}
\Xi_{k,\ell+1}=F_{\ell+1}
\quad
\text{on}
\quad
\overline{\Omega}_{\ell+1}\cup\ldots\cup\overline{\Omega}_k,
\end{equation}
\begin{equation}
\label{Eq6}
d(\Xi_{k,\ell+1},F_{\ell+1})<2^{-2k+(\ell+1)}.
\end{equation}
The construction of $\Xi_{k,\ell+1}$ is described below.

Recall that $F_{\ell+1}$ is obtained from $F_{\ell}$ by a modification of $F_{\ell}$
on the set
\begin{equation}
\label{Eq6.5}
W_{\ell}=\bigcup_{i=1}^{N_{\ell}} \overline{B}_{i\ell} \subset\Omega_{\ell}.
\end{equation}
Hence
\begin{equation}
\label{Eq7}
F_{\ell+1}=F_{\ell}
\quad
\text{on}
\quad
Q\setminus W_{\ell}.
\end{equation}
Since both $F_{\ell}$ and $F_{\ell+1}$ are homeomorphisms of $Q$, \eqref{Eq7} implies that
\begin{equation}
\label{Eq8}
F_{\ell+1}(W_{\ell}) = F_{\ell}(W_{\ell}).
\end{equation}
We define
\begin{equation}
\label{Eq9}
\Xi_{k,\ell+1} = \Xi_{k\ell}
\quad
\text{on}
\quad
Q\setminus W_{\ell}
\end{equation}
and we still need to define $\Xi_{k,\ell+1}$ on $W_{\ell}$.

We claim that it suffices to define $\Xi_{k,\ell+1}$ in such a way that
\begin{enumerate}
\item[(a)] $\Xi_{k,\ell+1}$ is a measure preserving diffeomorphism of $W_{\ell}$ onto
$F_{\ell}(W_{\ell})$ that agrees with $F_{\ell}$ in $W_{\ell}$ near the boundary of $W_{\ell}$,\\
\item[(b)] $\Xi_{k,\ell+1} = F_{\ell+1}$ on
$(\overline{\Omega}_{\ell+1}\cup\ldots\cup\overline{\Omega}_k)\cap W_{\ell}$,\\
\item[(c)]
$\displaystyle
\sup_{x\in W_{\ell}} \big|\Xi_{k,\ell+1}(x) - F_{\ell+1}(x)\big| +
\sup_{x\in F_{\ell+1}(W_{\ell})} \big|\Xi_{k,\ell+1}^{-1}(x) - F_{\ell+1}^{-1}(x)\big| < 2^{-2k+\ell}.
$
\end{enumerate}
Before we proceed with the construction, we will show that properties (a), (b) and (c) imply
\eqref{Eq4}, \eqref{Eq5} and \eqref{Eq6}.

By the induction hypothesis, $\Xi_{k\ell}$ satisfies \eqref{Eq2} so
$F_{\ell} = \Xi_{k\ell}$ in $W_{\ell}\subset\Omega_{\ell}$ (see \eqref{Eq6.5}).
Thus (a) implies that $\Xi_{k,\ell+1}$ is a measure preserving diffeomorphism in
$W_{\ell}$ that maps $W_{\ell}$ onto $\Xi_{k\ell}(W_{\ell})$ and
$\Xi_{k,\ell+1}$ agrees with $\Xi_{k\ell}$ in $W_{\ell}$ near the boundary of $W_{\ell}$.
Since also $\Xi_{k,\ell+1}=\Xi_{k\ell}$ in $Q\setminus W_{\ell}$ (see \eqref{Eq9}),
it follows that $\Xi_{k,\ell+1}$ is a measure preserving diffeomorphism of $Q$ which is
identity on $\partial Q$. This proves \eqref{Eq4}.

Observe that \eqref{Eq2}, \eqref{Eq7} and \eqref{Eq9} imply
$$
\Xi_{k,\ell+1} = \Xi_{k\ell} = F_{\ell}= F_{\ell+1}
\quad
\text{on}
\quad
(Q\setminus W_{\ell})\cap (\overline{\Omega}_{\ell}\cup\ldots\cup\overline{\Omega}_k).
$$
Since
$$
(\overline{\Omega}_{\ell+1}\cup\ldots\cup\overline{\Omega}_k)\setminus W_{\ell} \subset
(Q\setminus W_{\ell}) \cap (\overline{\Omega}_{\ell}\cup\ldots\cup\overline{\Omega}_k)
$$
we have that
$$
\Xi_{k,\ell+1} = F_{\ell+1}
\quad
\text{on}
\quad
(\overline{\Omega}_{\ell+1}\cup\ldots\cup\overline{\Omega}_k)\setminus W_{\ell}.
$$
This, along with the property (b), implies \eqref{Eq5}.

It follows from (a) and \eqref{Eq8} that
$\Xi_{k,\ell+1}(W_{\ell}) = F_{\ell+1}(W_{\ell})$. This yields
$$
\Xi_{k,\ell+1}(Q\setminus W_{\ell}) = F_{\ell+1}(Q\setminus W_{\ell}) = Q\setminus F_{\ell+1}(W_{\ell}).
$$
Now \eqref{Eq7} and \eqref{Eq9} imply that
$$
\Xi_{k,\ell+1}=\Xi_{k\ell}
\quad
\text{and}
\quad
F_{\ell+1}=F_{\ell}
\qquad
\text{in $Q\setminus W_{\ell}$.}
$$
Hence
$$
\Xi_{k,\ell+1}^{-1}=\Xi_{k\ell}^{-1}
\quad
\text{and}
\quad
F_{\ell+1}^{-1}=F_{\ell}^{-1}
\qquad
\text{in $Q\setminus F_{\ell+1}(W_{\ell})$.}
$$
Thus
\begin{eqnarray*}
\lefteqn{\sup_{x\in Q\setminus W_{\ell}} \big|\Xi_{k,\ell+1}(x) - F_{\ell+1}(x)\big| +
\sup_{x\in Q\setminus F_{\ell+1}(W_{\ell})} \big|\Xi_{k,\ell+1}^{-1}(x) - F_{\ell+1}^{-1}(x)\big|} \\
& = &
\sup_{x\in Q\setminus W_{\ell}} \big|\Xi_{k,\ell}(x) - F_{\ell}(x)\big| +
\sup_{x\in Q\setminus F_{\ell+1}(W_{\ell})} \big|\Xi_{k,\ell}^{-1}(x) - F_{\ell}^{-1}(x)\big| \\
& \leq &
d(\Xi_{k\ell},F_{\ell}) < 2^{-2k+\ell}
\end{eqnarray*}
by the induction hypothesis on $\Xi_{k\ell}$.

The above estimate along with (c) yields \eqref{Eq6} because of the following elementary inequality
$$
\sup_X|f|+\sup_X |g| \leq \Big(\sup_A |f|+\sup_B |g|\Big) + \Big(\sup_{X\setminus A} |f|+\sup_{X\setminus B}|g|\Big).
$$

We proved that (a), (b) and (c) imply \eqref{Eq4}, \eqref{Eq5} and \eqref{Eq6} so it remains to define
$\Xi_{k,\ell+1}$ in $W_{\ell}$ in a way that it will have properties (a), (b) and (c).

According to (a) we have to construct $\Xi_{k,\ell+1}$ in $W_{\ell}$ in a way that it will have the same image as $F_{\ell}$.
However, because of property (b), the diffeomorphism $\Xi_{k,\ell+1}$ will have to agree with $F_{\ell+1}$ on some subset of
$W_{\ell}$, so the construction of $\Xi_{k,\ell+1}$ will involve that of $F_{\ell+1}$.

Recall that $F_{\ell+1}$ is obtained from $F_{\ell}$ by a modification of $F_{\ell}$ on the set
$$
W_{\ell} = \bigcup_{i=1}^{N_{\ell}} \overline{B}_{\ell i} \subset\Omega_{\ell}
$$
and $F_{\ell+1}=F_{\ell}$ near the boundary of $W_{\ell}$, see \eqref{bimbala}.

In each ball we have a finite family of pairwise disjoint cubes $\mathcal{Q}_{\ell i} = \{ Q_{\ell i}^j\}_j$ such that
$F_{\ell+1}$ restricted to each of the cubes $Q_{\ell i}^j$ is a translation followed by a rescaled version of the homeomorphism
$\Phi$ from Lemma~\ref{base plus}. To emphasize this observation we will write
\begin{equation}
\label{Eq10}
F_{\ell+1}|_{Q_{\ell i}^j} = \Phi_{\ell i}^j.
\end{equation}
In each cube $Q_{\ell i}^j$ there is a Cantor set $A_{\ell i}^j$ which is translated and reflected by
$F_{\ell+1}|_{Q_{\ell i}^j} = \Phi_{\ell i}^j$.

Clearly, Corollary~\ref{000} applies to each of the mappings $\Phi_{\ell i}^j$ with $A_{\ell i}^j$
playing a role of $A$. Hence there are measure preserving diffeomorphisms
$$
\tilde{\Phi}_{\ell i}^j : Q_{\ell i}^j\to \Phi_{\ell i}^j (Q_{\ell i}^j)
$$
such that
\begin{equation}
\label{Eq11}
d(\tilde{\Phi}_{\ell i}^j,\Phi_{\ell i}^j)< 2^{-2k+\ell-1},
\end{equation}
$$
\tilde{\Phi}_{\ell i}^j=\Phi_{\ell i}^j
\quad
\text{near the boundary of $\partial Q_{\ell i}^j$ (i.e. $\tilde{\Phi}_{\ell i}^j$
is a translation near $\partial Q_{\ell i}^j$)},
$$
\begin{equation}
\label{Eq12}
\tilde{\Phi}_{\ell i}^j=\Phi_{\ell i}^j
\quad
\text{on}
\quad
(\overline{\Omega}_{\ell+1}\cup\ldots\cup\overline{\Omega}_k)\cap Q_{\ell i}^j.
\end{equation}

To see that the last condition can be guaranteed observe that, for $m=\ell+1,\ldots,k$,
$$
\overline{\Omega}_m\subset Q\setminus C_m\subset Q\setminus C_{\ell+1}\subset Q\setminus A_{\ell i}^j,
\quad
\text{since}
\quad
A_{\ell i}^j\subset\bigcup_{ij} A_{\ell i}^j=E_{\ell+1}\subset C_{\ell +1}.
$$
Thus
$$
\overline{\Omega}_{\ell+1}\cup\ldots\cup\overline{\Omega}_k\subset Q\setminus A_{\ell i}^j,
\quad
\text{so}
\quad
(\overline{\Omega}_{\ell+1}\cup\ldots\cup\overline{\Omega}_k)\cap Q_{\ell i}^j\subset Q_{\ell i}^j\setminus A_{\ell i}^j
$$
and hence Corollary~\ref{000} applies with the set
$(\overline{\Omega}_{\ell+1}\cup\ldots\cup\overline{\Omega}_k)\cap Q_{\ell i}^j$
playing the role of $K\subset Q\setminus A$. Therefore we can have $\tilde{\Phi}_{\ell i}^j$ satisfying \eqref{Eq12}.

Now we define $\Xi_{k,\ell+1}$ in $W_\ell$ as follows.
\begin{equation}
\label{Eq14.15}
\Xi_{k,\ell+1}(x)=
\left\{
\begin{array}{ccc}
\tilde{\Phi}_{\ell i}^j(x)    & \text{if} & \text{$x\in Q_{\ell i}^j$ for some $i,j$,}\\
F_{\ell+1}(x) &   \text{if}   & x\in W_{\ell}\setminus\bigcup_{ij} Q_{\ell i}^j.
\end{array}
\right.
\end{equation}
Observe that $F_{\ell+1}$ is a measure preserving diffeomorphism in $W_{\ell}\setminus \bigcup_{ij} Q_{\ell i}^j$.

It remains to prove that $\Xi_{k,\ell+1}$ has the properties (a), (b) and (c).

Note that
$$
\Xi_{k,\ell+1} = \tilde{\Phi}_{\ell i}^j
\quad
\text{on}
\quad
Q_{\ell i}^j
$$
and
$$
\tilde{\Phi}_{\ell i}^j = \Phi_{\ell i}^j = F_{\ell +1}
\quad
\text{near the boundary of $Q_{\ell i}^j$},
$$
by \eqref{Eq10} and Corollary~\ref{000}. This and the second line in \eqref{Eq14.15} imply that
$\Xi_{k,\ell+1}$ is a diffeomorphism in $W_{\ell}$. It is measure preserving because
$F_{\ell+1}$ and $\tilde{\Phi}_{\ell i}^j$ are measure preserving.

Since $\Xi_{k,\ell+1}=F_{\ell+1}=F_{\ell}$ near the boundary of $W_{\ell}$ (see \eqref{bimbala}),
$\Xi_{k,\ell+1}(W_{\ell})=F_{\ell}(W_{\ell})$. This proves the property (a).

Note that by \eqref{Eq10}, \eqref{Eq12} and the first line in \eqref{Eq14.15},
$$
\Xi_{k,\ell+1} = \tilde{\Phi}_{\ell i}^j = \Phi_{\ell i}^j = F_{\ell+1}
\quad
\text{on}
\quad
(\overline{\Omega}_{\ell+1}\cup\ldots\cup\overline{\Omega}_k)\cap Q_{\ell i}^j,
$$
so
$$
\Xi_{k,\ell+1} = F_{\ell+1}
\quad
\text{on}
\quad
(\overline{\Omega}_{\ell+1}\cup\ldots\cup\overline{\Omega}_k)\cap \Big(\bigcup_{ij}Q_{\ell i}^j\Big).
$$
Since also (see \eqref{Eq14.15})
\begin{equation}
\label{Eq16}
\Xi_{k,\ell+1} = F_{\ell+1}
\quad
\text{on}
\quad
W_{\ell}\setminus\bigcup_{ij} Q_{\ell i}^j,
\end{equation}
we conclude that
$$
\Xi_{k,\ell+1} = F_{\ell+1}
\quad
\text{on}
\quad
(\overline{\Omega}_{\ell+1}\cup\ldots\cup\overline{\Omega}_k)\cap W_{\ell},
$$
which is property (b).

Finally, \eqref{Eq16} shows that in order to prove (c) it suffices to show that
\begin{equation}
\label{Eq17}
\sup_{x\in\bigcup_{ij} Q_{\ell i}^j} \big|\Xi_{k,\ell+1}(x) - F_{\ell+1}(x)\big| +
\sup_{x\in\bigcup_{ij} F_{\ell+1}(Q_{\ell i}^j)} \big|\Xi_{k,\ell+1}^{-1}(x) - F_{\ell+1}^{-1}(x)\big| <
2^{-2k+\ell}.
\end{equation}
Observe that by \eqref{Eq10}, \eqref{Eq11} and the first line in \eqref{Eq14.15},
$$
\sup_{x\in Q_{\ell i}^j} \big|\Xi_{k,\ell+1}(x) - F_{\ell+1}(x)\big| +
\sup_{x\in F_{\ell+1}(Q_{\ell i}^j)} \big|\Xi_{k,\ell+1}^{-1}(x) - F_{\ell+1}^{-1}(x)\big| =
d(\tilde{\Phi}_{\ell i}^j,\Phi_{\ell i}^j) < 2^{-2k+\ell-1}.
$$
Hence \eqref{Eq17} follows from the elementary inequality
$$
\sup_{\bigcup_i A_i} |f| + \sup_{\bigcup B_i} |g| \leq
2\sup_i \Big(\sup_{A_i} |f|+\sup_{B_i} |g|\Big).
$$
The proof is complete.
\end{proof}

\section{Appendix}

The classical definition of an approximately differentiable function that can be found in most of the books
(see e.g. \cite{Whitney}) is provided below.
The aim of this Appendix is to show that this definition is equivalent with the one we used in the Introduction.
This is a folklore result, but we could not find a good reference for it. A similar result for approximate continuity on a real line can be found in \cite[Theorem 6.6]{LMZ}.

\begin{definition}[Classical definition]
Let $f:E\to\bbbr$ be a measurable function defined on a measurable set $E\subset\bbbr^n$. We say that
$f$ is {\em approximately differentiable} at $x\in E$ if there is a linear function $L:\bbbr^n\to\bbbr$
such that for any $\eps>0$ the set
\begin{equation}
\label{WT1-eq2}
\left\{ y\in E:\, \frac{|f(y)-f(x)-L(y-x)|}{|y-x|} <\eps \right\}
\end{equation}
has $x$ as a density point.
\end{definition}

\begin{proposition}
\label{WT1-T4}
A measurable function $f:E\to\bbbr$ defined in a measurable set $E\subset\bbbr^n$ is approximately
differentiable at $x\in E$ if and only if there is a measurable set $E_x\subset E$
and a linear function $L:\bbbr^n\to\bbbr$ such that $x$ is a density point of $E_x$ and
\begin{equation}
\label{WT1-eq3}
\lim_{E_x\ni y\to x} \frac{|f(y)-f(x)-L(y-x)|}{|y-x|} = 0.
\end{equation}
\end{proposition}
\begin{proof}
The implication from right to left is obvious: the set \eqref{WT1-eq2} contains
\mbox{$E_x\cap B(x,r)$} for some small $r$ and clearly $x$ is a density point of this set.
To prove the opposite implication we need to define the set $E_x$. Let $r_k$ be a
sequence strictly decreasing to $0$ such that
$$
r_{k+1}\leq \frac{r_k}{2^{k/n}}
$$
and
\begin{equation}
\label{WT1-eq4}
\left|\left\{ y\in B(x,r)\cap E:\, \frac{|f(y)-f(x)-L(y-x)|}{|y-x|}<\frac{1}{k}\right\}\right|
\geq \omega_n r^n\left(1-\frac{1}{2^k}\right)
\end{equation}
whenever $0<r\leq r_k$. Here $\omega_n$ stands for the volume of the unit ball in $\bbbr^n$. Let
$$
E_k=
\left\{ y\in B(x,r_k)\cap E:\, \frac{|f(y)-f(x)-L(y-x)|}{|y-x|}<\frac{1}{k}\right\}\, .
$$
It follows from \eqref{WT1-eq4} that
$$
|E_k|\geq \omega_n r_k^{n}\left(1-\frac{1}{2^k}\right).
$$
Finally, we define
$$
E_x=\bigcup_{k=1}^\infty (E_k\setminus B(x,r_{k+1})).
$$
The set $E_x$ is the union of the parts of the sets $E_k$ that are contained in the annuli
$B(x,r_k)\setminus B(x,r_{k+1})$. Clearly, the condition \eqref{WT1-eq3}
is satisfied and we only need to prove that $x$ is a density point of $E_x$.
If $r$ is small, then $r_{k+1}<r\leq r_k$ for some large $k$ and we need to show that
\begin{equation}
\label{WT1-eq5}
\frac{|B(x,r)\cap E_x|}{\omega_n r^n}\to 1
\quad
\mbox{as $k\to \infty$.}
\end{equation}
We have
\begin{eqnarray*}
\lefteqn{|B(x,r)\cap E_x|
 \geq
|(B(x,r)\cap E_k)\setminus B(x,r_{k+1})| + |E_{k+1}\setminus B(x,r_{k+2})|} \\
& \geq &
\left(\omega_n r^n\left(1-\frac{1}{2^k}\right) -\omega_n r_{k+1}^n\right)
+
\left(\omega_n r_{k+1}^n\left(1-\frac{1}{2^{k+1}}\right)-\omega_n r_{k+2}^n\right) \\
& = &
\omega_n r^n\left(1-\frac{1}{2^k}\right) - \frac{\omega_n r_{k+1}^n}{2^{k+1}}
-
\omega_n r_{k+2}^n
 >
\omega_n r^n\left(1-\frac{1}{2^{k-1}}\right)\,
\end{eqnarray*}
because
$$
r_{k+2}^n\leq \frac{r_{k+1}^n}{2^{k+1}}
\qquad
\mbox{and}
\qquad
r_{k+1}<r.
$$
Now \eqref{WT1-eq5} follows easily.
\end{proof}

\end{document}